\documentclass[10pt]{amsart}
\usepackage{amsmath,amscd}
\usepackage{amssymb}

\usepackage[all]{xy}

\usepackage[T1]{fontenc}

\usepackage{times}

\usepackage[T1]{fontenc}
\usepackage{graphicx}
\usepackage[svgnames]{xcolor}
\usepackage{framed}

\usepackage{tikz}

\usepackage{hyperref}
\hypersetup{urlcolor=blue, citecolor=blue, linkcolor=blue,hypertexnames=false}

\author{Liran Shaul}
\address{Universiteit Antwerpen, Departement Wiskunde-Informatica, Middelheim campus,
Middelheimlaan 1,
2020 Antwerp, Belgium}

\curraddr{Fakult\"at f\"ur Mathematik\\ 
Universit\"at Bielefeld\\ 
33501 Bielefeld\\ 
Germany.}

\email{LShaul@math.uni-bielefeld.de}
\newtheorem{thm}[equation]{Theorem}
\newtheorem{cor}[equation]{Corollary}
\newtheorem{prop}[equation]{Proposition}
\newtheorem{lem}[equation]{Lemma}

\theoremstyle{definition}
\newtheorem{dfn}[equation]{Definition}
\newtheorem{rem}[equation]{Remark}
\newtheorem{exa}[equation]{Example}

\newtheorem{notation}[equation]{Notation}

\newcommand{\iso}{\xrightarrow{\simeq}}

\newcommand{\opn}{\operatorname}
\newcommand{\cat}[1]{\operatorname{\mathsf{#1}}}

\newcommand{\mrm}[1]{\mathrm{#1}}
\newcommand{\mbb}[1]{\mathbb{#1}}

\newcommand{\til}[1]{\tilde{#1}}

\newcommand{\K}{\mbb{K} \hspace{0.05em}}

\newcommand{\thecat}{\opn{DGR}_{\opn{eftf}/\K}}

\newcommand{\heqv}{\mathrel{\overset{\makebox[0pt]{\mbox{\normalfont\tiny\sffamily h.e}}}{\iso}}}
\newcommand{\hequiv}{\heqv}

\newcommand{\eqfun}{\opn{EQ}}

\numberwithin{equation}{section}

\begin{document}

\title[The twisted inverse image pseudofunctor over DG rings]{The twisted inverse image pseudofunctor over commutative DG rings and perfect base change}

\thanks{The author acknowledges the support of the European Union for the ERC grant No 257004-HHNcdMir.}

\begin{abstract}
Let $\K$ be a Gorenstein noetherian ring of finite Krull dimension, and consider the category of cohomologically noetherian commutative differential graded rings $A$ over $\K$, such that $H^0(A)$ is essentially of finite type over $\K$, and $A$ has finite flat dimension over $\K$. We extend Grothendieck's twisted inverse image pseudofunctor to this category by generalizing the theory of rigid dualizing complexes to this setup.
We prove functoriality results with respect to cohomologically finite and cohomologically essentially smooth maps,
and prove a perfect base change result for $f^{!}$ in this setting. As application, we deduce a perfect derived base change result for the twisted inverse image of a map between ordinary commutative noetherian rings. Our results generalize and solve some recent conjectures of Yekutieli.
\end{abstract}

\maketitle

\setcounter{tocdepth}{1}
\tableofcontents

\setcounter{section}{-1}

\section{Introduction}

\subsection{Motivation: derived base change}

At the center of Grothendieck's coherent duality theory lies the twisted inverse image pseudofunctor $(-)^{!}$.
One of the most important features of $(-)^{!}$ is that it commutes with flat (and even tor-independent) base change.
In the affine noetherian situation that we work with in this paper, 
flat (or tor-independent) base change says the following:
given a diagram
\[
\xymatrixcolsep{4pc}
\xymatrix{
A\ar[r]^{f}\ar[d]^{g} & B\ar[d]^h\\
C \ar[r]^{f'} & B\otimes_A C
}
\]
of essentially finite type maps between noetherian rings,
such that $g$ is flat (or such that $\opn{Tor}^A_i(B,C) = 0$ for all $i\ne 0$),
there is an isomorphism
\[
\mrm{L}h^*f^{!}(-) \cong (f')^{!}\mrm{L}g^{*}(-).
\]
of functors $\cat{D}^{+}_{\mrm{f}}(A) \to \cat{D}^{+}_{\mrm{f}}(B\otimes_A C)$.
It is then natural to ask if one may generalize this result
with respect to more general maps $g$. 
Unfortunately, within the usual framework of noetherian rings and their derived categories,
such a question is meaningless. 
The problem is that if $\opn{Tor}^A_i(B,C) \ne 0$ for some $i \ne 0$,
the formula $M \mapsto M\otimes^{\mrm{L}}_A C$ does not define
a functor $\cat{D}(B) \to \cat{D}(B\otimes_A C)$.

The reason for the failure is that when $\opn{Tor}^A_i(B,C) \ne 0$,
the ring $B\otimes_A C$ does not represent the correct homological tensor product between $B$ and $C$ over $A$.
Instead, one needs to replace it with $B\otimes^{\mrm{L}}_A C$ which is no longer a commutative ring.
There are various ways to represent the object $B\otimes^{\mrm{L}}_A C$ (for instance, using simplicial commutative rings).
In this paper we will represent it using commutative differential graded rings (also known as graded-commutative DG-algebras).
Thus, we may resolve the map $g:A \to C$ as $A \xrightarrow{\til{g}} \til{C} \iso C$,
where $\til{g}$ is a flat DG-ring map, and the map $\til{C} \xrightarrow{c} C$ is a quasi-isomorphism.
We now obtain a new commutative diagram
\[
\xymatrixcolsep{4pc}
\xymatrix{
A\ar[r]^{f}\ar[d]^{\til{g}} & B\ar[d]^{\til{h}}\\
\til{C} \ar[r]^{\til{f}'} & B\otimes_A \til{C}
}
\]
in which $\til{g}$ and $\til{h}$ are flat,
so we may consider the functor
\[
\mrm{L}{\til{h}}^*f^{!}(-):\cat{D}^{+}_{\mrm{f}}(A) \to \cat{D}(B\otimes_A \til{C}),
\]
and ask whether it is naturally isomorphic to
\[
(\til{f}')^{!}\mrm{L}\til{g}^{*}(-).
\]
Of course, for this to make sense, 
one needs to define $(\til{f}')^{!}$. 
As $\til{C}$ and $B\otimes_A \til{C}$ are both DG-rings, to have such a result, 
it becomes necessary to extend the theory of the twisted inverse image pseudofunctor to such context.

\subsection{Rigid dualizing complexes}

There are various approaches in the literature concerning how to construct the twisted inverse image pseudofunctor.
Grothendieck's original strategy, explained by Hartshorne in \cite{RD}, 
used dualizing and residue complexes. 
Deligne (\cite[Appendix]{RD}) proved the existence of $f^{!}$ directly as a right adjoint of $\mrm{R}f_*$.
The most general results\footnote{See however the recent paper \cite{Ne2}, where duality theory is generalized to the unbounded derived category, at least over noetherian schemes} in the area are due to Lipman \cite{Li}
where such a pseudofunctor is constructed over non-noetherian schemes.
(And also, as far as we know, 
this is the first place where tor-independent base change was proved).
The approach taken in this paper is using rigid dualizing complexes.

First defined by Van den Bergh \cite{VdB} in a noncommutative situation,
the theory of rigid dualizing complexes,
while slightly less general then Grothendieck's approach,
has the advantage that it can be developed entirely inside the derived category.
In algebraic geometry, this theory was developed by
Yekutieli and Zhang \cite{Ye2,YZ1,YZ2,YZ3},
and by Avramov, Iyengar and Lipman \cite{AIL1,AIL2}, 
building on their work with Nayak \cite{AILN}.
From the point of view of this work,
it is interesting to note that while these papers developed
the theory of rigid dualizing complexes only over commutative rings,
both of them had to make a substantial use of DG-rings.
This is because DG-rings appear already in the definition of rigid dualizing complexes.
Underlying this definition, lies the following construction:
given a base commutative ring $\K$,
and a commutative noetherian $\K$-algebra $A$,
there is a functor
\[
\mrm{R}\opn{Hom}_{A\otimes^{\mrm{L}}_{\K} A}(A,-\otimes^{\mrm{L}}_{\K} -) :\cat{D}(A)\times \cat{D}(A) \to \cat{D}(A)
\]
which we refer to as derived Hochschild cohomology.
To define it, one replaces $\K \to A$ by a flat
DG-ring resolution $\K \to \til{A} \iso A$,
and then define
\[
\mrm{R}\opn{Hom}_{A\otimes^{\mrm{L}}_{\K} A}(A,-\otimes^{\mrm{L}}_{\K} -) := \mrm{R}\opn{Hom}_{\til{A}\otimes_{\K} \til{A}}(A,-\otimes^{\mrm{L}}_{\K} -).
\]
It can be shown (\cite[Theorem 3.2]{AILN}, \cite[Theorem 6.15]{Ye2}) that this construction is independent of the chosen resolution.
Then, one defines a rigid dualizing complex over $A$ relative to $\K$, 
to be a dualizing complex $R$ over $A$ which is of finite flat dimension over $\K$, 
together with an isomorphism
\[
R \iso \mrm{R}\opn{Hom}_{A\otimes^{\mrm{L}}_{\K} A}(A,R\otimes^{\mrm{L}}_{\K} R)
\]
in $\cat{D}(A)$.
\newline
Under suitable assumptions on $\K, A$, 
one can show that there exists a unique rigid dualizing complex $R_A$
over $A$ relative to $\K$.
Then, given an essentially finite type map $f:A \to B$,
one defines 
\[
f^{!}(-) := \mrm{R}\opn{Hom}_B(B\otimes^{\mrm{L}}_A \mrm{R}\opn{Hom}_A(-,R_A),R_B).
\]

\subsection{Main results}

Let us now describe the contents of this paper.
In section 1 we recall some basic definitions
concerning DG-rings, their derived categories,
and dualizing DG-modules over them. 
We also define the notion of a rigid dualizing DG-module over a commutative DG-ring,
and describe the category of DG-rings that we will work with in this paper, $\thecat$,
which consists of DG-rings which are essentially of finite type (in a cohomological sense)
and of finite flat dimension over a base noetherian ring $\K$.

Section 2 is of technical nature, 
and in it we describe a framework which allows one to
transfer homological results between derived categories of quasi-isomorphic DG-rings.
We view the results of this section as a technical necessity, forced on us due to the lack of a model structure on the category of commutative DG-rings.

In section 3 we study functoriality of rigid dualizing DG-modules with respect to cohomologically finite maps.
Such maps generalize finite ring maps. 
The main result of this section, 
repeated as Corollary \ref{cor:finite} below, is the following result.
\begin{thm}\label{thm04}
Let $\K$ be a Gorenstein noetherian ring of finite Krull dimension.
\begin{enumerate}
\item For any $A \in \thecat$, there exists a rigid dualizing DG-module over $A$ relative to $\K$.
\item Given a cohomologically finite map $f:A \to B$ in $\thecat$, 
and given a rigid dualizing DG-module $R_A$ over $A$ relative to $\K$, 
\[
R_B := \mrm{R}\opn{Hom}_A(B,R_A)
\]
is a rigid dualizing DG-module over $B$ relative to $\K$.
\end{enumerate} 
\end{thm}

In the short section 4 we discuss the box tensor product (that is, the external tensor product) of dualizing DG-modules.
We show in Theorem \ref{thm:tensor-of-dualizing} that:
\newpage
\begin{thm}
Let $\K$ be a Gorenstein noetherian ring of finite Krull dimension, 
let 
\[
A,B \in \thecat,
\] 
and let $C$ be a DG $\K$-ring that represents $A\otimes^{\mrm{L}}_{\K} B$.
Given a dualizing DG-module $R$ over $A$, and a dualizing DG-module $S$ over $B$, 
the DG-module $R\otimes^{\mrm{L}}_{\K} S$ is a dualizing DG-module over $C$.
\end{thm}

In section 5 we make a more detailed study of the behavior of derived Hochschild cohomology for $A \in \thecat$.
In Theorem \ref{thm:reduction}, we generalize \cite[Theorem 4.1]{AILN}, the main result of that paper, 
to $\thecat$, and give a formula for computing derived Hochschild cohomology using homological operations over $A$, 
avoiding the passage to $A\otimes^{\mrm{L}}_{\K} A$. 
Using that result, we deduce in Corollary \ref{cor:group} the following, 
which guarantees and explains the uniqueness of rigid dualizing DG-modules
\begin{thm}
Let $\K$ be a Gorenstein noetherian ring of finite Krull dimension.
For any $A \in \thecat$, 
denote by $\mathcal{D}_A$ the set of isomorphism classes of dualizing DG-modules over $A$.
Then the operation 
\[
\mrm{R}\opn{Hom}_{A\otimes^{\mrm{L}}_{\K} A}(A,-\otimes^{\mrm{L}}_{\K}-)
\]
defines a group structure on $\mathcal{D}_A$, 
and any rigid dualizing DG-module $(R_A,\rho)$ is a unit of this group.
In particular, the rigid dualizing DG-module is unique up to isomorphism.
\end{thm}

Section 6 discusses functoriality of rigid dualizing DG-modules with respect to cohomologically essentially smooth maps.
Such maps $A \to B$ generalizes the concept of an essentially smooth map between noetherian rings.
We associate to such a map a DG-module $\Omega_{B/A}$, 
which, when $A$ and $B$ are rings,
is given by a shift of a wedge product of the module of Kahler differentials $\Omega^1_{B/A}$.
Then, we show:
\begin{thm}\label{thm05}
Let $\K$ be a Gorenstein noetherian ring of finite Krull dimension, 
let $f:A \to B$ be a cohomologically essentially smooth map in $\thecat$,
and let $R_A$ be the rigid dualizing DG-module over $A$ relative to $\K$.
Then
\[
R_A \otimes^{\mrm{L}}_A \Omega_{B/A}
\]
has the structure of a rigid dualizing DG-module over $B$ relative to $\K$.
\end{thm}
This is repeated as Theorem \ref{thm:smooth} below.
To prove this, we first prove that $R_A \otimes^{\mrm{L}}_A \Omega_{B/A}$ is a dualizing DG-module over $B$.
This is equivalent to showing that a cohomologically essentially smooth map is Gorenstein,
and we prove this fact in Corollary \ref{cor:dual-smooth}. Theorems \ref{thm04} and \ref{thm05} solve a generalization of \cite[Conjecture 9.8]{Ye1} in the case where all DG-rings have bounded cohomology.

Finally, in section 7, we arrive to the twisted inverse image pseudofunctor.
Denoting by $\mathbf{DerCat}_{\K}$ the 2-category of $\K$-linear triangulated categories,
the next result is extracted from Proposition \ref{prop:pseudofunctor}, Theorem \ref{theorem:finite}
and Theorem \ref{theorem:smooth}.

\begin{thm}
Let $\K$ be a Gorenstein noetherian ring of finite Krull dimension.
There exists a pseudofunctor
\[
(-)^{!} : \thecat \to \mathbf{DerCat}_{\K}
\]
with the following properties:
\begin{enumerate}
\item On the full subcategory of $\thecat$ made of essentially finite type $\K$-algebras 
which are of finite flat dimension over $\K$, $(-)^{!}$ is naturally isomorphic to the classical twisted inverse image pseudofunctor.
\item Given a cohomologically finite map $f:A \to B$ in $\thecat$,
there is an isomorphism
\[
f^{!}(M) \cong \mrm{R}\opn{Hom}_A(B,M)
\]
of functors $\cat{D}^{+}_{\mrm{f}}(A) \to \cat{D}^{+}_{\mrm{f}}(B)$.
\item Given a cohomologically essentially smooth map $f:A \to B$ in $\thecat$,
there is an isomorphism
\[
f^{!}(M) \cong M\otimes^{\mrm{L}}_A \Omega_{B/A}
\]
of functors $\cat{D}^{+}_{\mrm{f}}(A) \to \cat{D}^{+}_{\mrm{f}}(B)$.
\end{enumerate}
\end{thm}

The final result of this paper fulfills the motivation of Section 0.1 by proving a derived base change result
with respect to perfect maps (i.e., maps of finite flat dimension).
Thus, in Theorem \ref{thm:base-change}, we show: 
\begin{thm}
Let $\K$ be a Gorenstein noetherian ring of finite Krull dimension,
let $f:A \to B$ be an arbitrary map in $\thecat$, 
and let $g:A \to C$ be a K-flat map in $\thecat$ 
such that $C$ has finite flat dimension over $A$.
Consider the induced base change commutative diagram
\[
\xymatrixcolsep{4pc}
\xymatrix{
A \ar[r]^f \ar[d]^g & B\ar[d]^{h}\\
C \ar[r]^{f'} & B\otimes_A C
}
\]
Then there is an isomorphism
\[
\mrm{L}h^* \circ f^{!}(-) \cong (f')^{!} \circ \mrm{L}g^{*}(-)
\]
of functors
\[
\cat{D}^{+}_{\mrm{f}}(A) \to \cat{D}^{+}_{\mrm{f}}(B\otimes_A C).
\]
\end{thm}

\subsection{Scope of this work}

We now discuss the scope of this work, compared to related papers on rigid dualizing complexes.
It was shown in \cite[Theorem 8.5.6]{AIL1} that over a base Gorenstein ring of finite Krull dimension,
an essentially of finite type algebra has a  rigid dualizing complex if and only if it has finite flat dimension over the base.
Since by Kawasaki's solution to Sharp's conjecture (\cite[Corollary 1.4]{Ka}),
every noetherian ring possessing a dualizing complex is a quotient of such Gorenstein ring,
in some sense, this is the strongest possible existence result.
The approach to rigid dualizing complexes taken in \cite{YZ1, YZ2, YZ3} is very different than the one of \cite{AIL1, AIL2}.
The main difference is that the latter presupposes the existence of global Grothendieck duality theory,
while the aim of the former is to develop the theory of rigid dualizing complex,
and use it as a self-contained approach to Grothendieck duality theory.

The approach we take in this paper in some sense combines these two approaches.
On one hand, we work in $\thecat$, the precise DG generalization of the scope of \cite{AIL1}.
In addition, we assume Grothendieck duality theory over ordinary commutative rings.
On the other hand, we cannot assume the existence of a twisted inverse image pseudofunctor over the category of DG-rings,
because as far as we know, no previous such work on this pseudofunctor worked in such a great generality as done here.
Thus, we will generalize here the main results of \cite{YZ1,YZ2} to $\thecat$,
and this will allow us to construct the twisted inverse image pseudofunctor on $\thecat$.

\subsection{Related papers}

We finish the introduction by discussing some related works concerning duality theory and dualizing DG-modules in derived algebraic geometry.
The first paper we are aware of that defined dualizing DG-modules over DG-rings was \cite{Hi}, 
where Hinich defined them over some specific DG-rings associated to a given local ring.
In the generality we work in this paper, 
dualizing DG-modules were first defined by Frankild, Iyengar and Jorgensen in \cite{FIJ},
where some basic properties of dualizing DG-modules were established.
The recent paper \cite{Ye1} by Yekutieli generalized the definition further,
and made a very detailed study of dualizing DG-modules in a very general commutative setting.
We will use \cite{Ye1} as our main reference for facts about dualizing DG-modules.
Finally, dualizing DG-modules were generalized even further to $E_{\infty}$-rings in the recent work of Lurie \cite{Lu}.

Regarding duality theory in derived algebraic geometry, 
we are aware of only one paper that constructed the twisted inverse image pseudofunctor in such a setting,
namely, the recent paper \cite{Ga} by Gaitsgory. 
In this paper, Gaitsgory extends duality theory to the category of DG-schemes which are of finite type over a field of characteristic zero.
To define $f^{!}$, Gaitsgory generalizes the approach of Deligne, defining $f^{!}$ as a right adjoint.
One the one hand, his work is more general, as he works with global DG-schemes.
On the other hand, in the affine situation, 
we of course work in a much more general setting, 
allowing a base Gorenstein ring, rather then a field,
and also considering essentially finite type maps instead of maps of finite type.
It is our expectation that the greater generality of this paper will be useful in arithmetic applications.

\section{Preliminaries}

We shall assume classical Grothendieck duality theory,
as developed in \cite{Co,RD,Li,Ne} and its extension to essentially finite type maps, 
explained in \cite{Na}.
We will mostly follow the notations of \cite{Ye1} concerning DG-rings and their derived categories. 
See \cite{Av,BS,Ke} for more background on commutative DG-rings.
Below is a short summary of the notation we use here.

A differential graded ring (abbrivated DG-ring),
is a $\mathbb{Z}$-graded ring 
\[
A = \bigoplus_{n=-\infty}^{\infty} A^n,
\]
together with a degree $+1$ differential $d:A \to A$, 
such that 
\[
d(a\cdot b) = d(a)\cdot b + (-1)^i\cdot a \cdot d(b), 
\]
for any $a \in A^i, b \in A^j$. 
We say $A$ is commutative if for any $a \in A^i, b \in A^j$, 
we have that $b\cdot a = (-1)^{i\cdot j} \cdot a \cdot b$, and $a\cdot a = 0 $ if $i$ is odd.
A DG-ring $A$ is said non-positive if $A^i = 0$ for all $i>0$.
\newline
\textbf{In this paper we assume that all DG-rings are commutative and non-positive.}\\

For such a DG-ring $A$, 
it is important to note that $A^0$ is an ordinary commutative ring.
In addition, the cohomology $H(A)$ of a DG-ring $A$ is a graded-ring,
and moreover, $H^0(A)$ is also a commutative ring, as it is a quotient of $A^0$.
We will set $\bar{A} := H^0(A)$. 
For any DG-ring $A$, there are homomorphisms of DG-rings $A^0 \to A$ and $A \to \bar{A}$.

Given a DG-ring $A$, we denote by $\opn{DGMod} A$ the category of DG-modules over it.
There is a forgetful functor $M \mapsto M^{\natural}$ which forgets the differential, 
and associates to a DG-module the underlying graded module.
The category $\opn{DGMod} A$ is an abelian category. By inverting quasi-isomorphisms in it,
one obtains its derived category, which we will denote by $\cat{D}(A)$. See \cite{Ke} for details about this construction.
We denote by $\cat{D}^{+}(A), \cat{D}^{-}(A)$ and $\cat{D}^{\mrm{b}}(A)$ the full triangulated subcategories of $\cat{D}(A)$ composed of DG-modules whose cohomologies are bounded below, bounded above or bounded.
Given a DG-module $M$ and an integer $n$, 
$\mrm{H}^n(M)$ has the structure of a $\bar{A} := H^0(A)$-module.
Assuming $\bar{A}$ is noetherian,
we say $M$ has finitely generated cohomologies if for each $n$,
$\mrm{H}^n(M)$ is finitely generated $\bar{A}$-module.
Again, under this noetherian assumption, 
we denote by $\cat{D}_{\mrm{f}}(A)$ the triangulated subcategory of $\cat{D}(A)$ made of DG-modules with finitely generated cohomology. 
As usual, one may combine these two finiteness conditions, and write, for example, $\cat{D}^{-}_{\mrm{f}}(A)$
for the triangulated subcategory made of DG-modules which are bounded above and have finitely generated cohomology.
For any map of DG-rings $f:A \to B$,
there is an associated forgetful functor $\opn{For}_f:\opn{DGMod} B \to \opn{DGMod} A$.
Since it is clearly exact, it also induces a functor 
$\cat{D}(B) \to \cat{D}(A)$, 
which, by abuse of notation, we also denote by $\opn{For}_f$.

We say that a DG-ring $A$ is cohomologically noetherian,
if $\bar{A}$ is a noetherian ring, 
$H(A)$ is bounded DG-module, 
and for each $i<0$, 
$H^i(A)$ is a finitely generated $\bar{A}$-module. 
It is clear that being cohomologically noetherian is preserved under quasi-isomorphisms.

Following \cite[Section 2]{Ye1}, we say that a DG-module $M \in \cat{D}(A)$ 
has finite flat dimension relative to $\cat{D}(A)$,
if there are integers $d_1,d_2$, such that for any DG-module $N$ with $\inf H(N) = e \ge -\infty$, $\sup H(N) = f \le \infty$,
we have that $\inf H(M\otimes^{\mrm{L}}_A N) \ge e-d_1$ and $\sup H(M\otimes^{\mrm{L}}_A N) \le f+d_2$.
One defines similarly the notion of finite injective dimension and finite projective dimension relative to $\cat{D}(A)$.
Given a cohomologically noetherian DG-ring $A$,
a dualizing DG-module over it is a DG-module $R \in \cat{D}^{\mrm{b}}_{\mrm{f}}(A)$
which has finite injective dimension relative to $\cat{D}(A)$,
and such that the canonical map $A \to \mrm{R}\opn{Hom}_A(R,R)$ is an isomorphism in $\cat{D}(A)$.
A tilting DG-module is a DG-module $P \in \cat{D}(A)$, such that
$P\otimes^{\mrm{L}}_A Q \cong A$ for some $Q \in \cat{D}(A)$.
See \cite[Section 6]{Ye1} for a discussion. 
Given two dualizing DG-modules $R_1, R_2$ over $A$,
by \cite[Theorem 7.10]{Ye1}, there is some tilting DG-module $P$, 
such that $R_1 \cong R_2 \otimes^{\mrm{L}}_A P$.

Given a commutative ring $\K$, and a $\K$ DG-ring $A$,
there is a functor
\[
\mrm{R}\opn{Hom}_{A\otimes^{\mrm{L}}_{\K} A}(A,-\otimes^{\mrm{L}}_{\K} -) : \cat{D}(A) \times \cat{D}(A) \to \cat{D}(A)
\]
which we refer to as derived Hochschild cohomology. 
Its existence is shown in \cite[Theorem 6.15]{Ye2} (in a much more general setting),
where it was denoted by $\opn{Rect}_{A/\K}(-,-)$.
Our notation follows \cite[Section 3]{AILN},
where a similar result (\cite[Theorem 3.2]{AILN}) was shown in the particular case where $A = A^0$. 

We now arrive to the main technical definition of this paper. 
It is an immediate generalization of \cite[Definition 4.1]{YZ1}.

\begin{dfn}
Let $\K$ be a noetherian ring, 
and let $A$ be a cohomologically noetherian DG $\K$-ring.
A pair $(R,\rho)$ is called a rigid DG-module over $A$ relative to $\K$,
if $R \in \cat{D}^{\mrm{b}}_{\mrm{f}}(A)$,
$R$ has finite flat dimension over $\K$,
and $\rho$ is an isomorphism
\[
\rho : R \iso \mrm{R}\opn{Hom}_{A\otimes^{\mrm{L}}_{\K} A}(A,R\otimes^{\mrm{L}}_{\K} R)
\]
in $\cat{D}(A)$. 
If in addition, $R$ is a dualizing DG-module over $A$, then
$(R,\rho)$ is called a rigid dualizing DG-module over $A$ relative to $\K$.
\end{dfn}

Let $\K$ be a noetherian ring, 
and let $A$ be a DG-ring over $\K$.
We say that $A$ is cohomologically essentially of finite type over $\K$
if $A$ is cohomologically noetherian,
and the composed map $\K \to A \to \bar{A}$ is essentially of finite type.
We denote by $\thecat$ the category of DG-rings $A$ which are cohomologically essentially of finite type over $\K$,
and such that $A$ has finite flat dimension over $\K$. 
Maps in this category are $\K$-linear maps of DG-rings.
Notice that $\thecat$ is closed under $\K$-linear quasi-isomorphisms. 
Moreover, the flat dimension assumption is precisely what is needed for $\thecat$ being closed under derived tensor product.
That is, if $A, B \in \thecat$, with one of them being K-flat over $\K$,
then $A\otimes_{\K} B \in \thecat$. 
\newline
In the rest of the paper, we will develop the theory of rigid dualizing DG-modules over DG-rings in $\thecat$,
where $\K$ is a fixed Gorenstein noetherian ring of finite Krull dimension.
\section{Homological equivalences between DG-rings}

A major advantage when working with DG-rings compared to ordinary rings is the existence of non-trivial quasi-isomorphisms.
Thus, given a DG-ring $A$ that we wish to study, 
we may replace it by a quasi-isomorphic DG-ring $B$ which might have a more convenient structure. 
We then use the fact that the derived categories of DG-modules over $A$ and $B$ are equivalent to transfer homological results from
$\cat{D}(B)$ to $\cat{D}(A)$. 
To make the transfer process, it is convenient to use the following notation: given a map $f:A\to B$ between two DG-rings, 
we denote by $\mrm{L}f^*$ and $f_*$ the functors
\[
\mrm{L}f^*(-) := -\otimes^{\mrm{L}}_A B : \cat{D}(A) \to \cat{D}(B), \quad
f_*(-) := \opn{For}_f(-) : \cat{D}(B) \to \cat{D}(A).
\]
\begin{prop}\label{prop:tensor-with-forget}
Let $f:A \to B$ be a quasi-isomorphism between DG rings. Then there are functorial isomorphisms
\[
B\otimes^{\mrm{L}}_A \opn{For}_f(-) = \mrm{L}f^{*}(f_*(-)) \cong 1_{\cat{D}(B)},
\]
\[
\opn{For}_f(B\otimes^{\mrm{L}}_A -) = f_{*}(\mrm{L}f^*(-)) \cong 1_{\cat{D}(A)},
\]
and
\[
\mrm{R}\opn{Hom}_A(B,f_*(-)) \cong 1_{\cat{D}(B)}.
\]
\end{prop}
\begin{proof}
This follows from \cite[Proposition 2.6]{Ye2}. 
\end{proof}

When constructing the derived category of a ring (or a DG-ring) as a localization at the class of quasi-isomorphisms, a priori a morphism between two complexes $M,N$ is given by a zig-zag
\[
\xymatrix{
 & C_1 & & C_3 & \dots & C_n\ar[rd]\\
 M \ar[ur] & & C_2 \ar[lu]\ar[ur] & & \dots \ar[lu] & & N
}
\]
in which all maps except possibly $C_n \to N$ are quasi-isomorphisms. 
It is a basic fact in the theory of derived categories, 
that one does not need to consider such a long chain, 
and that any morphism can be represented as
\[
\xymatrix{
 & C\ar[ld]\ar[rd] & \\
 M & & N
}
\]
where $C \to M$ is a quasi-isomorphism. 
Unfortunately, we do not know if such a reduction of zig-zags exists for maps of DG-rings,
so in general, in a multiplicative situation, we will work with long zig-zags. 
This motivates the next definitions.

\newpage

\begin{dfn}
\leavevmode
\begin{enumerate}
\item We say that two DG-rings $A$ and $B$ are homologically equivalent via a sequence $f=(f_1,\dots,f_{n-1})$, 
if there is a finite sequence of DG-rings 
\[
A = C_1, C_2,\dots,C_{n-1},C_n=B,
\]
and for each $1\le i \le n-1$, $f_i$ is either a quasi-isomorphism $f_i:C_i \to C_{i+1}$ or a quasi-isomorphism $f_i: C_{i+1} \to C_i$.
In this case, we write 
\[
f:A\heqv B. 
\]
\item Moreover, given a ring $\K$, if for each $i$, $C_i$ is a DG-ring over $\K$, and the maps $f_i:C_i \to C_{i+1}$ (or $f_i:C_{i+1} \to C_i$) are all $\K$-linear, then we say that $A$ and $B$ are homologically equivalent over $\K$ via $f$.
\end{enumerate}
\end{dfn}

A bit more generally, we will sometimes want to resolve by zig-zags a map of DG-rings.
\begin{dfn}\label{dfn:equiv-of-morphisms}
Let $\varphi:A \to B$ be a map between two DG-rings, and let 
\[
f=(f_1,\dots,f_{n-1}):A \heqv \til{A},\quad g=(g_1,\dots,g_{n-1}):B\heqv \til{B}
\]
be two homological equivalences of the same length, and such that for each $i$, 
the maps $f_i$ and $g_i$ are in the same direction.
Suppose we are given for $1\le i \le n$ a map of DG-rings $\psi_i:C_i \to D_i$,
with $\psi_1 = \varphi$, 
such that either there is a commutative diagram
\[
\xymatrix{
C_i \ar[r]^{f_i}\ar[d]^{\psi_i} & C_{i+1}\ar[d]^{\psi_{i+1}}\\
D_i \ar[r]^{g_i} & D_{i+1}
}
\]
or there is a commutative diagram
\[
\xymatrix{
C_i \ar[d]^{\psi_i} & C_{i+1}\ar[l]_{f_i}\ar[d]^{\psi_{i+1}}\\
D_i  & D_{i+1}\ar[l]_{g_i}
}
\]
Then we say that $\varphi: A\to B$ and $\chi := \psi_{n}: \til{A} \to \til{B}$ are homologically equivalent via $\psi = (\psi_1,\dots,\psi_{n})$, and write $\psi:\varphi \heqv \chi$.
\end{dfn}
It is clear from the definitions that if $A \in \thecat$, and if $A$ and $B$ are homologically equivalent over $\K$, then $B \in \thecat$.

Given a homological equivalence $f:A\heqv B$ over $\K$, we define an equivalence of categories
\[
\eqfun_f:\cat{D}(A) \to \cat{D}(B)
\]
as follows: first, if we are given a quasi-isomorphism $f_i:C_i \to C_{i+1}$, then we define $\eqfun_i := \mrm{L}(f_i)^*$, 
while if we are given a quasi-isomorphism $f_i:C_{i+1} \to C_i$, we define $\eqfun_i := (f_i)_*$.
Now, we set 
\[
\eqfun_f := \eqfun_{n-1} \circ \eqfun_{n-2} \circ \dots \circ \eqfun_2 \circ \eqfun_1 : \cat{D}(A) \to \cat{D}(B).
\]
By Proposition \ref{prop:tensor-with-forget}, we see that for each $i$, $\eqfun_i$ is a $\K$-linear equivalence, 
so that $\eqfun_f$ is also a $\K$-linear equivalence. 
It is clear from the definition of $\eqfun_f$ that there is an isomorphism
\begin{equation}\label{equiv-of-the-dgr}
\eqfun_f(A) \cong B
\end{equation}
in $\cat{D}(B)$.

Note also that if $A$ and $B$ are homologically equivalent via $f$ over $\K$, then $B$ and $A$ are homologically equivalent via 
\[
f^{-1} := (f_{n-1},\dots,f_1)
\]
over $\K$, and using Proposition \ref{prop:tensor-with-forget} again, we see that $\eqfun_f$ and $\eqfun_{f^{-1}}$ are quasi-inverse to each other. 
We now show that all the main constructions of this paper are essentially invariant with respect to such homological equivalences.

\begin{prop}\label{prop:transfer}
Let $\K$ be a ring, and let $f:A\heqv B$ be an homological equivalence over $\K$.
\begin{enumerate}
\item There is an isomorphism
\[
\mrm{R}\opn{Hom}_{A\otimes^{\mrm{L}}_{\K} A}(A,M\otimes^{\mrm{L}}_{\K} N) \cong
\eqfun_{f^{-1}}( \mrm{R}\opn{Hom}_{B\otimes^{\mrm{L}}_{\K} B}(B,\eqfun_f(M)\otimes^{\mrm{L}}_{\K} \eqfun_f(N)) )
\]
of functors 
\[
\cat{D}(A)\times\cat{D}(A)\to\cat{D}(A).
\]
\item If $A$ (and hence $B$) is cohomologically noetherian, and $R_A$ is a dualizing DG-module over $A$, then $\eqfun_f(R_A)$ is a dualizing DG-module over $B$. 
\item If $R_A$ has the structure of a rigid DG-module over $A$ relative to $\K$, 
then $\eqfun_f(R_A)$ has the structure of a rigid DG-module over $B$ relative to $\K$.
\item There is an isomorphism
\[
\mrm{R}\opn{Hom}_A(M,N) \cong
\eqfun_{f^{-1}}(\mrm{R}\opn{Hom}_B(\eqfun_f(M),\eqfun_f(N))) 
\]
of functors 
\[
\cat{D}(A)\times\cat{D}(A)\to\cat{D}(A).
\]
\end{enumerate}
\end{prop}
\begin{proof}
\leavevmode
\begin{enumerate}
\item Let $M,N \in \cat{D}(A)$. Step 1: Suppose first that $f:A\to B$ is a single quasi-isomorphism. Let $\K \to \til{A} \xrightarrow{g} A$ be a K-flat resolution of $\K \to A$. 
Then since $f$ is a quasi-isomorphism,
\[
\K \to \til{A} \xrightarrow{h} B
\]
is a K-flat resolution of $\K \to B$, where $h = f\circ g$.
Hence, by \cite[Theorem 6.15]{Ye2}, there is a functorial isomorphism
\begin{eqnarray}
\eqfun_{f^{-1}}( \mrm{R}\opn{Hom}_{B\otimes^{\mrm{L}}_{\K} B}(B,\eqfun_f(M)\otimes^{\mrm{L}}_{\K} \eqfun_f(N)) ) =\nonumber\\
f_*(\mrm{R}\opn{Hom}_{B\otimes^{\mrm{L}}_{\K} B}(B,\mrm{L}f^*(M)\otimes^{\mrm{L}}_{\K} \mrm{L}f^*(N))) \cong\nonumber\\
f_*(\mrm{R}\opn{Hom}_{\til{A}\otimes_{\K} \til{A}}(B,h_*(\mrm{L}f^*(M))\otimes^{\mrm{L}}_{\K} h_*(\mrm{L}f^*(N)))).\nonumber
\end{eqnarray}
Note that $h_* = (f\circ g)_* \cong g_* \circ f_*$, so that $h_*(\mrm{L}f^*(-)) \cong g_*(-)$. Also, it is clear that
\[
f_*(\mrm{R}\opn{Hom}_{\til{A}\otimes_{\K} \til{A}}(B,-)) \cong \mrm{R}\opn{Hom}_{\til{A}\otimes_{\K} \til{A}}(A,-),
\]
and using \cite[Theorem 6.15]{Ye2} again, we see that
\[
\mrm{R}\opn{Hom}_{A\otimes^{\mrm{L}}_{\K} A}(A,M\otimes^{\mrm{L}}_{\K} N) \cong 
\mrm{R}\opn{Hom}_{\til{A}\otimes_{\K} \til{A}}(A,g_*(M)\otimes^{\mrm{L}}_{\K} g_*(N)),
\]
so in this case we get the required functorial isomorphism.

Step 2: Assume now that that $f:B\to A$ is a single quasi-isomorphism in the other direction.
Let $\K \to \til{B} \xrightarrow{g} B$ be a K-flat resolution of $\K \to B$. Then, 
\[
\K \to \til{B} \xrightarrow{h} A
\]
is K-flat resolution of $\K \to A$, where $h = f \circ g$. 
By \cite[Theorem 6.15]{Ye2},
\begin{eqnarray}
\eqfun_{f^{-1}}( \mrm{R}\opn{Hom}_{B\otimes^{\mrm{L}}_{\K} B}(B,\eqfun_f(M)\otimes^{\mrm{L}}_{\K} \eqfun_f(N)) ) =\nonumber\\
\mrm{L}f^* ( \mrm{R}\opn{Hom}_{B\otimes^{\mrm{L}}_{\K} B}(B,f_*(M)\otimes^{\mrm{L}}_{\K} f_*(N)))\cong\nonumber\\
\mrm{L}f^*( \mrm{R}\opn{Hom}_{\til{B}\otimes_{\K}\til{B}}(B,g_*(f_*(M)) \otimes^{\mrm{L}}_{\K} g_*(f_*(N)))\cong\nonumber\\
\mrm{L}f^*( \mrm{R}\opn{Hom}_{\til{B}\otimes_{\K}\til{B}}(B,h_*(M) \otimes^{\mrm{L}}_{\K}h_*(N))),\nonumber
\end{eqnarray}
On the other hand,
\begin{eqnarray}
\mrm{R}\opn{Hom}_{A\otimes^{\mrm{L}}_{\K} A}(A,M\otimes^{\mrm{L}}_{\K} N) \cong\nonumber\\
\mrm{R}\opn{Hom}_{\til{B}\otimes_{\K}\til{B}}(A,h_*(M) \otimes^{\mrm{L}}_{\K}h_*(N))\cong\nonumber\\
\mrm{L}f^*(f_*( \mrm{R}\opn{Hom}_{\til{B}\otimes_{\K}\til{B}}(A,h_*(M) \otimes^{\mrm{L}}_{\K}h_*(N)) ) )\cong\nonumber\\
\mrm{L}f^*( \mrm{R}\opn{Hom}_{\til{B}\otimes_{\K}\til{B}}(B,h_*(M) \otimes^{\mrm{L}}_{\K}h_*(N))),\nonumber
\end{eqnarray}
which proves the claim in this case.
\newline
Step 3: The general case is now easy to verify by induction, using steps 1,2, and using the isomorphisms $f_* \mrm{L}f^* \cong 1$ and $\mrm{L}f^* f_* \cong 1$.
\item Let $R_A$ be a dualizing DG-module over $A$. It is clear that $\eqfun_f$ preserves boundedness, finitely generated cohomology, and finite injective dimension. If $f:B\to A$ is a single quasi-isomorphism, the fact that $f_*(R_A)$ is a dualizing DG-module is shown in \cite[Proposition 7.5(2)]{Ye1}. If $f:A \to B$ is a single quasi-isomorphism, we have that
\begin{eqnarray}
\mrm{R}\opn{Hom}_B(\mrm{L}f^*(R_A),\mrm{L}f^*(R_A)) \cong\nonumber\\ \mrm{L}f^*(f_*(\mrm{R}\opn{Hom}_B(\mrm{L}f^*(R_A),\mrm{L}f^*(R_A)) )) \cong\nonumber\\
\mrm{L}f^* \mrm{R}\opn{Hom}_A(R_A,R_A) \cong \mrm{L}f^*A = B.\nonumber
\end{eqnarray}
For a general homological equivalence $f$ between $A$ and $B$, the fact that $\eqfun_f(R_A)$ is a dualizing DG-module over $B$ follows now by induction. 

\item Assume that $R_A$ has the structure of a rigid DG-module over $A$ relative to $\K$. Set $R_B := \eqfun_f(R_A)$.
When considered as objects of $\cat{D}(\K)$, $R_A$ and $R_B$ are isomorphic, so $R_B$ has finite flat dimension over $\K$.
Since $\eqfun_f$ and $\eqfun_{f^{-1}}$ are quasi-inverse to each other, we have that
\begin{eqnarray}
\mrm{R}\opn{Hom}_{B\otimes^{\mrm{L}}_{\K} B}(B,R_B\otimes^{\mrm{L}}_{\K} R_B) \cong\nonumber\\
\eqfun_f ( \eqfun_{f^{-1}} ( 
\mrm{R}\opn{Hom}_{B\otimes^{\mrm{L}}_{\K} B}(B,\eqfun_f(R_A)\otimes^{\mrm{L}}_{\K} \eqfun_f(R_A))
) ),\nonumber
\end{eqnarray}
and by part (1) of this proposition, this is naturally isomorphic to
\[
\eqfun_f(\mrm{R}\opn{Hom}_{A\otimes^{\mrm{L}}_{\K} A}(A,R_A\otimes^{\mrm{L}}_{\K} R_A) ) \cong \eqfun_f(R_A) = R_B,
\]
where the last isomorphism follows from the rigidity of $R_A$.
\item Identical to the proof of (2) above.
\end{enumerate}
\end{proof}

\begin{prop}\label{prop:transfer-finite}
Let $\varphi:A \to B$ and $\chi:\til{A}\to \til{B}$ be maps of DG-rings, 
and let $\psi:\varphi\hequiv \chi$ be an homological equivalence, 
and denote by $f:A \hequiv \til{A}$ and $g:B \hequiv \til{B}$ the homological equivalences underlying it.
Then there is an isomorphism
\[
\eqfun_{g^{-1}}(\mrm{R}\opn{Hom}_{\til{A}}(\til{B},\eqfun_f(-))) \cong \mrm{R}\opn{Hom}_A(B,-)
\]
of functors 
\[
\cat{D}(A) \to \cat{D}(B).
\]
\end{prop}
\begin{proof}
Step 1: Assume first that $f:A\iso \til{A}$ and $g:B \iso \til{B}$ are single quasi-isomorphisms.
Unwrapping the definitions of $\eqfun$, we must show that there is a functorial isomorphism
\[
\opn{For}_g ( \mrm{R}\opn{Hom}_{\til{A}}(\til{B},-\otimes^{\mrm{L}}_A \til{A})) \cong \mrm{R}\opn{Hom}_A(B,-).
\]
By the last isomorphism of Proposition \ref{prop:tensor-with-forget}, 
given $M \in \cat{D}(A)$, we have a natural isomorphism
\[
M\otimes^{\mrm{L}}_A \til{A} \cong \mrm{R}\opn{Hom}_A(\til{A},\opn{For}_f(M\otimes^{\mrm{L}}_A \til{A})) \cong \mrm{R}\opn{Hom}_A(\til{A},M)
\]
in $\cat{D}(\til{A})$.
Hence, using adjunction,
\begin{eqnarray}
\opn{For}_g ( \mrm{R}\opn{Hom}_{\til{A}}(\til{B},M\otimes^{\mrm{L}}_A \til{A})) \cong\nonumber\\
\opn{For}_g ( \mrm{R}\opn{Hom}_{\til{A}}(\til{B},\mrm{R}\opn{Hom}_A(\til{A},M)))\cong\nonumber\\
\opn{For}_g( \mrm{R}\opn{Hom}_A(\til{B},M)) = \mrm{R}\opn{Hom}_A(B,M).\nonumber
\end{eqnarray}
Step 2:
Assume now that $f:\til{A} \iso A$ and $g:\til{B} \iso B$ are single quasi-isomorphisms in the other direction.
We must show that
\[
B\otimes^{\mrm{L}}_{\til{B}} \mrm{R}\opn{Hom}_{\til{A}}(\til{B},\opn{For}_f(M)) \cong \mrm{R}\opn{Hom}_A(B,M).
\]
As in step 1, we have a functorial isomorphism
\begin{eqnarray}
B\otimes^{\mrm{L}}_{\til{B}} \mrm{R}\opn{Hom}_{\til{A}}(\til{B},\opn{For}_f(M)) \cong \nonumber\\
\mrm{R}\opn{Hom}_{\til{B}}(B, \mrm{R}\opn{Hom}_{\til{A}}(\til{B},\opn{For}_f(M))) \cong\nonumber\\
\mrm{R}\opn{Hom}_{\til{A}}(B,\opn{For}_f(M)).\nonumber
\end{eqnarray}
Taking K-injective resolutions $\opn{For}_f(M) \iso \til{I}$ and $M \iso I$, 
it is thus enough to show that there is a $B$-linear quasi-isomorphism
\[
\opn{Hom}_A(B,I) \to \opn{Hom}_{\til{A}}(B,\til{A}),
\]
but as $I$ and $\til{I}$ are isomorphic in $\cat{D}(\til{A})$, 
there is a quasi-isomorphism $I \to \til{I}$ in $\cat{D}(\til{A})$,
so this follows from \cite[Proposition 2.6(2)]{Ye2}. 
Step 3: again, the general case follows immediately by induction.
\end{proof}

Having established that the main constructions of this paper are preserved along homological equivalences, 
we now show that objects of $\thecat$ are homologically equivalent to nice objects. More precisely:

\begin{prop}\label{prop:res} 
Let $\K$ be a noetherian ring, and let $A\in \thecat$. Then there is a DG-ring $\til{A} \in \thecat$ which satisfies the following:
\begin{enumerate}
\item The DG-rings $A$ and $\til{A}$ are homologically equivalent over $\K$.
\item One has $\til{A}^0 = S^{-1}(\K[x_1,\dots,x_n])$, for some $n \in \mathbb{N}$, and for some multiplicatively closed set $S \subseteq \K[x_1,\dots,x_n]$.
\item For each $i<0$, the $\til{A}^0$-module $\til{A}^i$ is finitely generated and free. 
\item In particular, $\til{A}$ is K-projective over $\K$.
\end{enumerate}
\end{prop}
\begin{proof}
In the proof of \cite[Lemma 7.8]{Ye1}, the author constructs DG-rings $A_{\opn{loc}}$ and $A_{\opn{eft}}$ with $\K$-linear quasi-isomorphisms
\[
A \rightarrow A_{\opn{loc}} \leftarrow A_{\opn{eft}}.
\]
Setting $\til{A}:= A_{\opn{eft}}$, we see that $A$ is homologically equivalent to $\til{A}$ over $\K$, 
and it is easy to see that the DG-ring $\til{A}$ constructed there satisfies the above.
\end{proof}

We end this section with a couple of useful facts that use the above and will be useful in the sequel.
A map $f:A \to B$ of DG-rings is called cohomologically finite if the induced map $\bar{f}:\bar{A} \to \bar{B}$ is a finite ring map.

\begin{prop}
Let $\K$ be a noetherian ring which has dualizing complexes, and let $A \in \thecat$. Then $A$ has dualizing DG-modules.
\end{prop}
\begin{proof}
Let $\til{A}$ be the DG-ring which cohomologically equivalent to $A$ over $\K$, constructed in Proposition \ref{prop:res}. 
By Proposition \ref{prop:transfer}, it is enough to show that $\til{A}$ has dualizing DG-modules. 
The ring $\til{A}^0$ is essentially of finite type over $\K$, so it has dualizing complexes. 
The map $\til{A}^0 \to \til{A}$ is cohomologically finite. 
Hence, the result follows from \cite[Proposition 7.5(1)]{Ye1}.
\end{proof}

\begin{prop}\label{prop:ffd}
Let $\K$ be a Gorenstein noetherian ring of finite Krull dimension, let $A \in \thecat$, and let $R$ be a dualizing DG-module over $A$.
Then $R$ has finite flat dimension over $\K$.
\end{prop}
\begin{proof}
Let $\til{A}$ be the DG-ring which is cohomologically equivalent to $A$ over $\K$, constructed in Proposition \ref{prop:res}.
By Proposition \ref{prop:transfer}, there is a dualizing DG-module $S$ over $\til{A}$, which is isomorphic to $R$ in $\cat{D}(\K)$, 
so it is enough to show that $S$ has finite flat dimension over $\K$. Denote by $f$ the structure map $f:\K \to \til{A}$, and let
\[
\omega := (f^0)^!(\K),
\]
where $f^0:\K\to \til{A}^0$ is the zero component of $f$. 
Since $\K$ is a dualizing complex over itself, $\omega$ is a dualizing complex over $\til{A}^0$. 
Moreover, in the terminology of \cite[Section 1]{AILN}, it is the relative dualizing complex of $f^0$. 
By our assumptions, $\til{A}$ has bounded cohomology, and for each $i$, $H^i(\til{A})$ is finitely generated over $H^0(\til{A})$, 
so a fortiori, it is finitely generated over $\til{A}^0$. Thus, $\til{A} \in \cat{D}^{\mrm{b}}_{\mrm{f}}(\til{A}^0)$. 
Further, we assumed that $\til{A}$ has finite flat dimension over $\K$. 
Hence, in the terminology of \cite[Section 1]{AILN}, we have that $\til{A} \in \mathsf{P}(f^0)$. 
By \cite[Theorem 1.2(1)]{AILN}, this implies that
\[
S' := \mrm{R}\opn{Hom}_{\til{A}^0}(\til{A},\omega) \in \mathsf{P}(f^0),
\]
and in particular $S'$ has finite flat dimension over $\K$.
The map $\til{A}^0 \to \til{A}$ is clearly cohomologically finite, so by \cite[Proposition 7.5(1)]{Ye1}, 
$S'$ is a dualizing DG-module over $\til{A}$. 
Hence, by \cite[Theorem 7.10(2)]{Ye1}, there is some tilting DG-module $P$ over $\til{A}$, such that 
\[
S \cong S' \otimes^{\mrm{L}}_{\til{A}} P.
\]
By \cite[Theorem 6.5]{Ye1} and \cite[Theorem 5.11(iii)]{Ye1}, 
$P$ has finite flat dimension over $\til{A}$, 
so we get that $S$ (and hence $R$) also has finite flat dimension over $\K$.
\end{proof}

\newpage

\section{Rigid dualizing DG-modules and cohomologically finite maps}

Recall that a map $f:A\to B$ between two DG-rings is called cohomologically finite
if the induced map $\bar{f}:\bar{A} \to \bar{B}$ is a finite ring map.

\subsection{Splitting lemmas} 

Let $\K$ be a ring, let $A,B$ be $\K$-algebras, let $P,M \in \opn{Mod} A$, and let $Q,N \in \opn{Mod} B$. There is a map
\[
\opn{Hom}_A(P,M) \otimes_{\K} \opn{Hom}_B(Q,N) \to \opn{Hom}_{A\otimes_{\K} B}(P\otimes_{\K} Q,M\otimes_{\K} N),
\]
given by 
\[
\phi\otimes_{\K} \psi \mapsto ( (p\otimes_{\K} q) \mapsto (\phi(p)\otimes_{\K} \psi(q)).
\]
This map is clearly $A\otimes_{\K} B$-linear, and is functorial in $P,M,Q,N$. 
If we assume that $P$ and $Q$ are finitely generated and projective, then it is easy to see that it is an isomorphism. 
This short technical subsection is dedicated for studying derived versions of this natural isomorphism. 
Such isomorphisms will play an important role in the sequel.

Now, assume that $A$ and $B$ are DG-rings over $\K$ (but $\K$ is still a ring!), that  $P,M \in \opn{DGMod} A$, 
and that $Q,N \in \opn{DGMod} B$. Then it is clear that (up to signs) the same formula defines a map
\[
\opn{Hom}_A(P,M) \otimes_{\K} \opn{Hom}_B(Q,N) \to \opn{Hom}_{A\otimes_{\K} B}(P\otimes_{\K} Q,M\otimes_{\K} N),
\]
in $\opn{DGMod} A\otimes_{\K} B$, which again, is functorial in $P,M,Q,N$. 
If $P^{\natural} \cong (A^{\natural})^{\oplus r}$ and $Q^{\natural} \cong (B^{\natural})^{\oplus s}$, where $r,s \in \mathbb{N}$, then this map is an isomorphism.

\begin{lem}\label{lem:tensor-split}
Let $\K$ be a ring, let $A,B$ be cohomologically noetherian DG-rings over $\K$, and let $P,M \in \cat{D}(A)$  and $Q,N \in \cat{D}(B)$.
Assume that the following holds:
\begin{enumerate}
\item For each $i\le 0$, $A^i$ is flat over $\K$.
\item We have that $P \in \cat{D}^{-}_{\mrm{f}}(A)$, $Q \in \cat{D}^{-}_{\mrm{f}}(B)$, $M \in \cat{D}^{\mrm{b}}(A)$,
$N \in \cat{D}^{\mrm{b}}(B)$.
\item When considered as objects of $\cat{D}(\K)$, the complexes $M$ and $\mrm{R}\opn{Hom}_B(Q,N)$ have finite flat dimension over $\K$.
\end{enumerate}
Then there is an isomorphism
\[
\mrm{R}\opn{Hom}_A(P,M) \otimes^{\mrm{L}}_{\K} \mrm{R}\opn{Hom}_B(Q,N) \cong \mrm{R}\opn{Hom}_{A\otimes_{\K} B}(P\otimes^{\mrm{L}}_{\K} Q,M\otimes^{\mrm{L}}_{\K} N)
\]
in $\cat{D}(A\otimes_{\K} B)$, functorial in $P,M,Q,N$.
\end{lem}
\begin{proof}
Since $A$ and $B$ are cohomologically noetherian, using \cite[Proposition 1.18(2)]{Ye1}, we may replace $P$ and $Q$ by pseudo-finite semi-free resolutions $P_f\iso P$ and $Q_f\iso Q$. By \cite[Proposition 1.14(1)]{Ye1}, we have that
\[
P_f^{\natural} \cong \bigoplus_{i=-\infty}^{i_1} A^{\natural}[-i]^{\oplus r_i}   ,\quad
Q_f^{\natural} \cong \bigoplus_{j=-\infty}^{j_1} B^{\natural}[-j]^{\oplus s_j}, \quad r_i,s_j<\infty.
\]
This implies that $P_f \otimes_{\K} Q_f$ is also of this form, so using \cite[Proposition 1.14(1)]{Ye1} again, we see that it is a pseudo-finite semi-free resolution of $P\otimes^{\mrm{L}}_{\K} Q$. In particular, $P_f$ is K-projective over $A$, $Q_f$ is K-projective over $B$, and $P_f\otimes_{\K} Q_f$ is K-projective over $A\otimes_{\K} B$.

Next, replace $M$  by a semi-free resolution $M' \iso M$ over $A$. Because $A$ is made of flat $\K$-modules, $M'$ is a bounded above complex of flat modules, when considered as a complex over $\K$. By assumption, $M'$ has finite flat dimension over $\K$. Hence, by \cite[Theorem 2.4.F, (i) $\implies$ (vi)]{AF}, for $n$ small enough, the smart truncation map $M' \to \opn{smt}^{\ge n}(M)$, which is $A$-linear, is a quasi-isomorphism, and $M'' := \opn{smt}^{\ge n}(M)$ is a bounded DG-module over $A$, made of flat $\K$-modules. 

Because $P_f$ and $Q_f$ are pseudo-finite semi-free and $M''$ and $N$ are bounded, 
we deduce that the natural map
\begin{eqnarray}
\opn{Hom}_A(P_f,M'') \otimes_{\K} \opn{Hom}_B(Q_f,N) \to\label{eqn:same-without-der}\\
\opn{Hom}_{A\otimes_{\K} B}(P_f\otimes_{\K} Q_f,M'' \otimes_{\K} N).\nonumber
\end{eqnarray}
is an isomorphism.

Finally, notice that, when considered as complexes of $\K$-modules, $\opn{Hom}_A(P_f,M'')$ is a bounded below complex of flat $\K$-modules, and by assumption $\opn{Hom}_B(Q_f,N)$ has finite flat dimension over $\K$. Hence, by \cite[Lemma 5.4]{YZ1}, the natural map
\begin{eqnarray}
\opn{Hom}_A(P_f,M'') \otimes^{\mrm{L}}_{\K} \opn{Hom}_B(Q_f,N) \to\label{eqn:derived-not-needed}\\
\opn{Hom}_A(P_f,M'') \otimes_{\K} \opn{Hom}_B(Q_f,N)\nonumber
\end{eqnarray}
is a quasi-isomorphism. The result now follows from combining (\ref{eqn:same-without-der}) and (\ref{eqn:derived-not-needed}).
\end{proof}

We shall also need the following variation of the above.

\begin{lem}\label{lem:tensor-split-variation}
Let $\K$ be a coherent commutative ring (e.g, a noetherian ring), let $A,B,C,D$ be cohomologically noetherian DG-rings over $\K$, and let $M \in \cat{D}^{\mrm{b}}(A)$ and $N \in \cat{D}^{\mrm{b}}(B)$. Let $\varphi:A \to C$ and $\psi:B \to D$ be two  DG $\K$-ring homomorphisms.
Assume that the following holds:
\begin{enumerate}
\item For each $i\le 0$, $A^i$ is flat over $\K$.
\item The maps $\varphi,\psi$ are semi-free and cohomologically finite.
\item When considered as objects of $\cat{D}(\K)$, the complexes $M$ and $\mrm{R}\opn{Hom}_B(D,N)$ have finite flat dimension over $\K$.
\end{enumerate}
Then there is an isomorphism
\[
\mrm{R}\opn{Hom}_A(C,M) \otimes^{\mrm{L}}_{\K} \mrm{R}\opn{Hom}_B(D,N) \cong \mrm{R}\opn{Hom}_{A\otimes_{\K} B}(C\otimes_{\K} D,M\otimes^{\mrm{L}}_{\K} N)
\]
in $\cat{D}(C\otimes_{\K} D)$, functorial in $M,N$.
\end{lem}

\begin{proof}
Replace $M$ by $M''$ as in the proof of Lemma \ref{lem:tensor-split}.
Note that since $C$ is semi-free over $A$, 
it is K-projective over $A$, 
so that
\[
\mrm{R}\opn{Hom}_A(C,M) \cong \opn{Hom}_A(C,M) \cong \opn{Hom}_A(C,M'').
\]
Because $C$ is semi-free over $A$, 
the complex of $\K$-modules $\opn{Hom}_A(C,M'')$
is a bounded below complex made of (possibly infinite) direct product of flat $\K$-modules.
Since $\K$ is coherent, an arbitrary direct product of flats is flat, 
so $\opn{Hom}_A(C,M'')$ is a bounded below complex of flat $\K$-modules.
Hence, as in the proof of Lemma \ref{lem:tensor-split}, 
the canonical $C\otimes_{\K} D$-linear map 
\[
\mrm{R}\opn{Hom}_A(C,M'') \otimes^{\mrm{L}}_{\K} \mrm{R}\opn{Hom}_B(D,N) \to \mrm{R}\opn{Hom}_A(C,M'') \otimes_{\K} \mrm{R}\opn{Hom}_B(D,N)
\]
is a quasi-isomorphism.

The cohomological finiteness assumptions on $A \to C$ and $B \to D$ 
and the fact that $C$ and $D$ are cohomologically noetherian,
imply that $C \in \cat{D}^{-}_{\mrm{f}}(A)$,
and $D \in \cat{D}^{-}_{\mrm{f}}(B)$.
Using \cite[Proposition 1.18(2)]{Ye1},
let $A \to P \iso C$ and $B \to Q \iso D$ be pseudo-finite semi-free resolutions 
in $\cat{D}(A)$ and $\cat{D}(B)$ respectively. 

Consider the following commutative diagram induced by these maps:
\tiny
\[
\xymatrix{
\opn{Hom}_A(C,M'') \otimes^{\mrm{L}}_{\K} \opn{Hom}_B(D,N) \ar[r]\ar[d] & \opn{Hom}_A(C,M'') \otimes_{\K} \opn{Hom}_B(D,N) \ar[r]\ar[d] &
\opn{Hom}_{A\otimes_{\K} B}(C\otimes_{\K} D,M''\otimes_{\K} N)\ar[d]\\
\opn{Hom}_A(P,M'') \otimes^{\mrm{L}}_{\K} \opn{Hom}_B(Q,N) \ar[r] & \opn{Hom}_A(P,M'') \otimes_{\K} \opn{Hom}_B(Q,N) \ar[r] &
\opn{Hom}_{A\otimes_{\K} B}(P\otimes_{\K} Q,M''\otimes_{\K} N)
}
\]
\normalsize

The leftmost vertical map is clearly a quasi-isomorphism.
Also, we have seen that the two leftmost horizontal maps are quasi-isomorphisms.
Hence, the middle vertical map is also a quasi-isomorphism.
The rightmost vertical map is also clearly a quasi-isomorphism, 
and because of finiteness of $P$ and $Q$, the right lower horizontal map is also a quasi-isomorphism.
Hence, we deduce that the right upper horizontal map is a quasi-isomorphism.
\newline
It follows that the two $C\otimes_{\K} D$-linear maps
\begin{eqnarray} 
\opn{Hom}_A(C,M'') \otimes^{\mrm{L}}_{\K} \opn{Hom}_B(D,N) \to\nonumber\\ \opn{Hom}_A(C,M'') \otimes_{\K} \opn{Hom}_B(D,N) \to\nonumber\\
\opn{Hom}_{A\otimes_{\K} B}(C\otimes_{\K} D,M''\otimes_{\K} N)\nonumber
\end{eqnarray}
are quasi-isomorphisms. This proves the claim.
\end{proof}

\begin{rem}
Let $\K$ be a coherent ring, and let $A,B,C,D$ be arbitrary cohomologically noetherian DG-rings over $\K$.
Let $A \to C$ and $B \to D$ be any cohomologically finite DG $\K$-ring homomorphisms. 
Then one can always choose resolutions 
\[
\K \to \til{A} \iso A, \quad \til{A} \to \til{C} \iso C, \quad B \to \til{D} \iso D,
\]
such that $\til{A},B,\til{C},\til{D}$ satisfy the conditions of the above lemma. 
Hence, the lemma essentially represents the fact that whenever $A \to C$ and $B \to D$ are cohomologically finite, 
under modest finiteness assumptions on $M,N$, there is a natural isomorphism
\[
\mrm{R}\opn{Hom}_A(C,M) \otimes^{\mrm{L}}_{\K} \mrm{R}\opn{Hom}_B(D,N) \cong \mrm{R}\opn{Hom}_{A\otimes^{\mrm{L}}_{\K} B}(C\otimes^{\mrm{L}}_{\K} D,M\otimes^{\mrm{L}}_{\K} N)
\]
as functors 
\[
\cat{D}^{\mrm{b}}(A) \times \cat{D}^{\mrm{b}}(B) \to \cat{D}(C\otimes^{\mrm{L}}_{\K} D).
\]
It is an open problem to the author if given a cohomologically finite DG $\K$-ring homomorphism $A \to C$, 
it can be resolved as $A \to \til{C} \iso C$, where $\til{C}$ is pseudo-finite semi-free over $A$. 
This is known if $C$ itself is a ring (\cite[Proposition 2.2.8]{Av} when both $A$ and $C$ are rings, and \cite[Proposition 1.7(3)]{YZ1} when $A$ is a DG-ring and $C$ is a ring),
but the proofs of this fact fail when $C \ne C^0$, 
thus we needed the workaround of the above lemma.
\end{rem}

\subsection{Cohomologically finite maps}

\begin{dfn}
Given a cohomologically finite map $f:A \to B$ between two cohomologically noetherian DG-rings, 
we denote by $f^{\flat}$ the functor
\[
f^{\flat}(-) := \mrm{R}\opn{Hom}_A(B,-) : \cat{D}(A) \to \cat{D}(B).
\]
\end{dfn}
For our next result, we shall need a technical refinement of Proposition \ref{prop:res}.

\begin{lem}\label{lem:var-res}
Let $\K$ be a noetherian ring, and let $A \in \thecat$.
Let $A\to B$ be a map between DG $\K$-rings, and consider the diagram
\[
A \rightarrow A_{\opn{loc}} \leftarrow A_{\opn{eft}} = \til{A}
\]
constructed in the proof of Proposition \ref{prop:res}.
Then there is a $\K$-linear quasi-isomorphism $B \to B_{\opn{loc}}$ making the diagram
\[
\xymatrix{
A \ar[r]\ar[d] & A_{\opn{loc}} \ar[d] & A_{\opn{eft}}\ar[l] & \\
B \ar[r] & B_{\opn{loc}}
}
\]
commutative.
\end{lem}
\begin{proof}
The DG-ring $A_{\opn{loc}}$ was defined in \cite[Lemma 7.8]{Ye1} as follows: 
let $\pi_A: A\to \bar{A}$ be the canonical surjection, 
and let $S$ be the set of elements $s$ in $A^0$ such that $\pi_A(s)$ is invertible in $\bar{A}$.
It is clear that $S$ is multiplicatively closed, and one defines 
\[
A_{\opn{loc}} := A\otimes_{A^0} S^{-1}A^0.
\]
It is then easy to see that the map $A \to A_{\opn{loc}}$ induced by the localization map $A^0 \to S^{-1}A^0$
is a quasi-isomorphism, simply because $S^{-1}A^0 \otimes_{A^0} \bar{A} = \bar{A}$.
The map $A \to B$ is a map of DG-rings, so in particular, letting $\pi_B:B \to \bar{B}$ be the canonical surjection, 
we obtain a commutative diagram
\[
\xymatrix{
A \ar[r]^{\pi_A}\ar[d] & \bar{A}\ar[d]\\
B \ar[r]^{\pi_B} & \bar{B}
}
\]
But this implies that $S^{-1}A^0 \otimes_{A^0} \bar{B} = \bar{B}$, 
so letting $B_{\opn{loc}} := B\otimes_{A^0} S^{-1}A^0$, 
the map $B \to B_{\opn{loc}}$ induced by the localization map $A^0 \to S^{-1}A^0$
is again a quasi-isomorphism, and this proves the claim.
\end{proof}

Here is the main result of this section. 
It is a far reaching generalization of \cite[Theorem 5.3]{YZ1}.

\begin{thm}\label{thm:finite}
Let $\K$ be a noetherian ring, and let $f:A \to B$ be a cohomologically finite map in $\thecat$.
Let $(R_A,\rho)$ be a rigid DG-module over $A$ relative to $\K$, and assume that
\[
R_B := f^{\flat}(R_A)
\]
has finite flat dimension over $\K$.
Then $R_B$ has the structure of a rigid DG-module over $B$ relative to $\K$.
If moreover $(R_A,\rho)$ is a rigid dualizing DG-module over $A$ relative to $\K$, then
$R_B$ has the structure of a rigid dualizing DG-module over $B$ relative to $\K$.
\end{thm}
\begin{proof}
Let $\til{A}$ be the DG-ring which is homologically equivalent to $A$ constructed in Proposition \ref{prop:res}.
By Lemma \ref{lem:var-res}, there is a commutative diagram
\begin{equation}\label{eqn:diagram-in-finite}
\xymatrix{
A \ar[r]\ar[d] & A_{\opn{loc}} \ar[d] & \til{A}\ar[l] & \\
B \ar[r] & B_{\opn{loc}}
}
\end{equation}
such that all horizontal maps are quasi-isomorphisms, 
and for each $i\le 0$, $\til{A}^i$ is flat over $\K$.
Denote by $\eqfun_{\varphi}$ the equivalence of categories
\[
\eqfun_{\varphi} : \cat{D}(A) \to \cat{D}(\til{A})
\]
induced by this homological equivalence. 
Then by Proposition \ref{prop:transfer}, 
\[
R_{\til{A}} := \eqfun_{\varphi}(R_A)
\]
has the structure of a rigid DG-module over $\til{A}$ relative to $\K$.
Let us complete the diagram (\ref{eqn:diagram-in-finite}) to a commutative diagram
\[
\xymatrix{
A \ar[r]\ar[d] & A_{\opn{loc}} \ar[d] & \til{A}\ar[l]\ar[d] & \\
B \ar[r] & B_{\opn{loc}} & B_{\opn{loc}} \ar[l]
}
\]
where the right bottom map is the identity. 
Denoting by $\psi$ the homological equivalence $\psi:B \hequiv B_{\opn{loc}}$, 
we see that $\varphi$ and $\psi$ are homological equivalent in the sense of Definition \ref{dfn:equiv-of-morphisms}.

Notice that the map $\til{A} \to B_{\opn{loc}}$ in (\ref{eqn:diagram-in-finite}) is also cohomologically finite.
Let $\til{A} \to \til{B} \iso B_{\opn{loc}}$ be a semi-free resolution of $\til{A} \to B_{\opn{loc}}$.
In particular, the map $\til{A} \to \til{B}$ is also cohomologically finite.
Hence, the conditions of Lemma \ref{lem:tensor-split-variation} are satisfied, and there is a natural isomorphism
\[
\mrm{R}\opn{Hom}_{\til{A}}(\til{B},R_{\til{A}}) \otimes^{\mrm{L}}_{\K} \mrm{R}\opn{Hom}_{\til{A}}(\til{B},R_{\til{A}}) \cong
\mrm{R}\opn{Hom}_{\til{A} \otimes_{\K} \til{A}}(\til{B}\otimes_{\K} \til{B}, R_{\til{A}}\otimes^{\mrm{L}}_{\K} R_{\til{A}})
\]
in $\cat{D}(\til{B}\otimes_{\K} \til{B})$.

Because $\til{A}$ is K-flat over $\K$, and $\til{B}$ is K-projective over $\til{A}$, 
we see that the sequence of maps $\K \to \til{B} \iso B_{\opn{loc}}$ is a K-flat resolution of $\K \to B_{\opn{loc}}$.

Hence, using \cite[Theorem 6.15]{Ye2}, we may calculate:
\begin{eqnarray}
\mrm{R}\opn{Hom}_{B_{\opn{loc}} \otimes^{\mrm{L}}_{\K} B_{\opn{loc}}}
(B_{\opn{loc}}, \mrm{R}\opn{Hom}_{\til{A}}(B_{\opn{loc}},R_{\til{A}}) \otimes^{\mrm{L}}_{\K} \mrm{R}\opn{Hom}_{\til{A}}(B_{\opn{loc}},R_{\til{A}}) )\cong\nonumber\\
\mrm{R}\opn{Hom}_{\til{B} \otimes_{\K} \til{B}}
(B_{\opn{loc}}, \mrm{R}\opn{Hom}_{\til{A}}(\til{B},R_{\til{A}}) \otimes^{\mrm{L}}_{\K} \mrm{R}\opn{Hom}_{\til{A}}(\til{B},R_{\til{A}}) ) \cong \nonumber\\
\mrm{R}\opn{Hom}_{\til{B} \otimes_{\K} \til{B}}
(B_{\opn{loc}}, 
\mrm{R}\opn{Hom}_{\til{A} \otimes_{\K} \til{A}}(\til{B}\otimes_{\K} \til{B}, R_{\til{A}}\otimes^{\mrm{L}}_{\K} R_{\til{A}})
)\cong\nonumber\\
\mrm{R}\opn{Hom}_{\til{A} \otimes_{\K} \til{A}}(B_{\opn{loc}}, R_{\til{A}}\otimes^{\mrm{L}}_{\K} R_{\til{A}})\cong\nonumber
\mrm{R}\opn{Hom}_{\til{A}}(B_{\opn{loc}},\mrm{R}\opn{Hom}_{\til{A}\otimes_{\K} \til{A}} (\til{A},R_{\til{A}}\otimes^{\mrm{L}}_{\K} R_{\til{A}})),\nonumber
\end{eqnarray}
where the last two maps are adjunctions. 
Because $R_{\til{A}}$ is a rigid DG-module over $\til{A}$ relative to $\K$, there is an isomorphism
\[
\mrm{R}\opn{Hom}_{\til{A}}(B_{\opn{loc}},\mrm{R}\opn{Hom}_{\til{A}\otimes_{\K} \til{A}} (\til{A},R_{\til{A}}\otimes^{\mrm{L}}_{\K} R_{\til{A}})) \cong
\mrm{R}\opn{Hom}_{\til{A}}(B_{\opn{loc}},R_{\til{A}}),
\]
and since 
\[
R_B \cong \mrm{R}\opn{Hom}_{\til{A}}(B_{\opn{loc}},R_{\til{A}})
\]
in $\cat{D}(\K)$, so that the latter also has finite flat dimension over $\K$, 
we deduce that 
\[
\mrm{R}\opn{Hom}_{\til{A}}(B_{\opn{loc}},R_{\til{A}})
\]
is a rigid DG-module over $B_{\opn{loc}}$ relative to $\K$.
Hence, by Proposition \ref{prop:transfer},
\[
\eqfun_{\psi^{-1}} (\mrm{R}\opn{Hom}_{\til{A}}(B_{\opn{loc}},R_{\til{A}}))
\]
is a rigid DG-module over $B$ relative to $\K$, 
and by Proposition \ref{prop:transfer-finite}
\[
\eqfun_{\psi^{-1}} (\mrm{R}\opn{Hom}_{\til{A}}(B_{\opn{loc}},R_{\til{A}})) \cong \mrm{R}\opn{Hom}_A(B,R_A),
\]
which proves the result.
Finally, if in addition $R_A$ is a dualizing DG-module over $A$, 
then by \cite[Proposition 7.5(1)]{Ye1}, $R_B$ is a dualizing DG-module over $B$. This completes the proof.
\end{proof}

\begin{cor}\label{cor:finite}
Let $\K$ be a Gorenstein noetherian ring of finite Krull dimension.
\begin{enumerate}
\item For any $A \in \thecat$, there exists a rigid dualizing DG-module over $A$ relative to $\K$.
\item Given a cohomologically finite map $f:A \to B$ in $\thecat$, 
and given a rigid dualizing DG-module $R_A$ over $A$ relative to $\K$, 
\[
R_B := f^{\flat}(R_A)
\]
is a rigid dualizing DG-module over $B$ relative to $\K$.
\end{enumerate} 
\end{cor}
\begin{proof}
\leavevmode
\begin{enumerate}
\item Let $A \rightarrow A_{\opn{loc}} \leftarrow A_{\opn{eft}}$ be the $\K$-linear homological equivalence between
$A$ and $A_{\opn{eft}}$ constructed in Proposition \ref{prop:res}. 
By Proposition \ref{prop:transfer},
it is enough to show that there exists a rigid dualizing DG-module over $A_{\opn{eft}}$ relative to $\K$.
Note that 
\[
A_{\opn{eft}}^0 = S^{-1}(\K[x_1,\dots,x_n]),
\]
so in particular it is flat and essentially of finite type over $\K$.
Hence, by \cite[Theorem 8.5.6]{AIL1} (or by \cite[Theorem 3.6]{YZ2} if $\K$ is regular), 
there exists a rigid dualizing complex $R$ over $A_{\opn{eft}}^0$ relative to $\K$.
Set 
\[
R_A := \mrm{R}\opn{Hom}_{A_{\opn{eft}}^0}(A_{\opn{eft}},R).
\]
Because $A_{\opn{eft}}^0 \to A_{\opn{eft}}$ is cohomologically finite,
$R_A$ is a dualizing DG-module over $A_{\opn{eft}}$, 
so by Proposition \ref{prop:ffd}, it has finite flat dimension over $\K$. 
Hence, we may apply Theorem \ref{thm:finite} to the map $A_{\opn{eft}}^0 \to A_{\opn{eft}}$,
and deduce that $R_A$ is a rigid dualizing DG-module over $A_{\opn{eft}}$ relative to $\K$.
\item Using Theorem \ref{thm:finite}, all we need to show is that $R_B$ has finite flat dimension over $\K$.
But $R_B$ is dualizing over $B$, so this follows from Proposition \ref{prop:ffd}.
\end{enumerate}
\end{proof}

\section{Tensor product of dualizing DG-modules}

The aim of this section is to show that the property of being a dualizing DG-module is preserved under the derived tensor product operation. 
In addition to being an interesting result on its own, 
it will be cruical in the proof of Theorem \ref{thm:reduction} below.
In accordance with the general theme of this paper, we avoid making any flatness assumptions. 
Hence, for the constructions to make sense, we shall need the following definition, 
which one might view as the most general way to define the derived tensor product of two DG-rings over a DG-ring.

\begin{dfn}
Given a DG-ring $\K$, and two DG $\K$-rings $A,B$, 
we will say that a DG $\K$-ring $C$ represents $A\otimes^{\mrm{L}}_{\K} B$ 
if $C$ is homologically equivalent over $\K$ to a DG $\K$-ring of the form $\til{A}\otimes_{\K} \til{B}$,
where $\K \to \til{A} \iso A$ and $\K \to \til{B} \iso B$ are DG-resolutions of $\K \to A$ and $\K \to B$ respectively, 
and at least one of the maps $\K \to \til{A}$ or $\K \to \til{B}$ is K-flat.
\end{dfn}

\begin{rem}
Given a DG-ring $\K$, two DG $\K$-rings $A,B$, 
and a DG $\K$-ring $C$ that represents $A\otimes^{\mrm{L}}_{\K} B$, the functor
\[
-\otimes^{\mrm{L}}_{\K} - : \cat{D}(A)\times \cat{D}(B) \to \cat{D}(C)
\]
is well defined. It can be defined as follows.
Assume $C$ is homologically equivalent via some $f$ over $\K$ to $\til{A}\otimes_{\K} \til{B}$, 
and without loss of generality suppose that $\til{A}$ is K-flat over $\K$.
Denote by $\varphi$ and $\psi$ the quasi-isomorphisms 
$\til{A} \xrightarrow{\varphi} A$ and $\til{B} \xrightarrow{\psi} B$.
Given $M \in \cat{D}(A)$ and $N \in \cat{D}(B)$, 
let $P \iso \opn{For}_{\varphi}(M)$ be a K-flat resolution in $\cat{D}(\til{A})$, and define
\[
\cat{D}(C) \ni M \otimes^{\mrm{L}}_{\K} N := \eqfun_{f^{-1}} ( P\otimes_{\K} \opn{For}_{\psi}(N) ).
\]
Here, $\eqfun_{f^{-1}}$ is the equivalence of categories $\cat{D}(\til{A}\otimes_{\K} \til{B}) \to \cat{D}(C)$ that was defined in the discussion preceding Proposition \ref{prop:transfer}.
\end{rem}

\begin{lem}\label{lem:box-tensor-of-tilting}
Let $A \to B$ and $A \to C$ be homomorphisms of DG rings,
with $A \to C$ being K-flat. 
Given a tilting DG-module $L$ over $B$, 
and a tilting DG-module $N$ over $C$, 
the DG-module $L\otimes^{\mrm{L}}_{A} N$ is a tilting DG-module over $B\otimes_{A} C$.
\end{lem}
\begin{proof}
Letting $L' := \mrm{R}\opn{Hom}_B(L,B)$, and let $N' := \mrm{R}\opn{Hom}_C(N,C)$. By \cite[Corollary 6.6]{Ye1}, there are isomorphisms
\[
L\otimes^{\mrm{L}}_B L' \cong B
\]
and
\[
N\otimes^{\mrm{L}}_C N' \cong C,
\]
so the result follows from the isomorphism
\[
(L\otimes^{\mrm{L}}_{A} N) \otimes^{\mrm{L}}_{B\otimes_A C} (L'\otimes^{\mrm{L}}_{A} N') \cong
(L\otimes^{\mrm{L}}_B L') \otimes^{\mrm{L}}_{A} (N\otimes^{\mrm{L}}_C N').
\]
\end{proof}

\begin{thm}\label{thm:tensor-of-dualizing}
Let $\K$ be a Gorenstein noetherian ring of finite Krull dimension, 
let 
\[
A,B \in \thecat,
\] 
and let $C$ be a DG $\K$-ring that represents $A\otimes^{\mrm{L}}_{\K} B$.
Given a dualizing DG-module $R$ over $A$, and a dualizing DG-module $S$ over $B$, 
the DG-module $R\otimes^{\mrm{L}}_{\K} S$ is a dualizing DG-module over $C$.
\end{thm}
\begin{proof}
Let $\K \to \til{A} \iso A$ and $\K \to \til{B} \iso B$ be resolutions as in Proposition \ref{prop:res}.
It is clear that $C$ is homologically equivalent over $\K$ to $\til{A}\otimes_{\K} \til {B}$, 
so by Proposition \ref{prop:transfer}, we may as well assume that $C = \til{A}\otimes_{\K} \til {B}$. 
In addition, by abuse of notation, we will replace $R$ and $S$ by 
their images in $\cat{D}(\til{A})$ and $\cat{D}(\til{B})$ under the corresponding forgetful functors.

Denote by $f:\K \to \til{A}$ and $g:\K \to \til{B}$ the structure maps, 
and set
\[
R^0 := (f^0)^{!}(\K), \quad S^0 := (g^0)^{!}(\K).
\]
These are dualizing complexes over $\til{A}^0$ and $\til{B}^0$ respectively. 
Further, recalling that $\til{A}^0$ and $\til{B}^0$ are localizations of polynomial rings over $\K$, 
we see that $R^0$ and $S^0$ are concretely given by some shifts of $\til{A}^0$ and $\til{B}^0$.
The ring $\til{A}^0 \otimes_{\K} \til{B}^0$ is also a localization of a polynomial ring over $\K$.
Hence, $R^0 \otimes_{\K} S^0$ is a dualizing complex over $\til{A}^0 \otimes_{\K} \til{B}^0$.
The maps $\til{A}^0 \to \til{A}$ and $\til{B}^0 \to \til{B}$ are cohomologically finite, 
so by \cite[Proposition 7.5(1)]{Ye1}, 
\[
R' := \mrm{R}\opn{Hom}_{\til{A}^0}(\til{A},R^0), \quad S':= \mrm{R}\opn{Hom}_{\til{B}^0}(\til{B},S^0)
\]
are dualizing DG-modules over $\til{A}$ and $\til{B}$ respectively.

We now claim that by Lemma \ref{lem:tensor-split-variation}, there is a natural isomorphism
\begin{eqnarray}
\mrm{R}\opn{Hom}_{\til{A}^0}(\til{A},R^0) \otimes^{\mrm{L}}_{\K} \mrm{R}\opn{Hom}_{\til{B}^0}(\til{B},S^0) \to\label{split-in-tdual}\\
\mrm{R}\opn{Hom}_{\til{A}^0 \otimes_{\K} \til{B}^0}( \til{A}\otimes_{\K} \til{B}, R^0 \otimes^{\mrm{L}}_{\K} S^0)\nonumber
\end{eqnarray}
in $\cat{D}(\til{A}\otimes_{\K} \til{B})$. One may easily see that all the conditions of this lemma are satisfied by this datum. 
The only non-trivial thing to check is that $S' = \mrm{R}\opn{Hom}_{\til{B}^0}(\til{B},S^0)$ has finite flat dimension over $\K$,
and this follows from Proposition \ref{prop:ffd}. 

Thus, (\ref{split-in-tdual}) is an isomorphism. 
But on its right hand side, because $\til{A}^0\otimes_{\K} \til{B}^0 \to \til{A}\otimes_{\K} \til{B}$ is cohomologically finite, 
\[
\mrm{R}\opn{Hom}_{\til{A}^0 \otimes_{\K} \til{B}^0}( \til{A}\otimes_{\K} \til{B}, R^0 \otimes^{\mrm{L}}_{\K} S^0)
\]
is a dualizing DG-module over $\til{A}\otimes_{\K} \til{B}$. 
Hence, $R'\otimes^{\mrm{L}}_{\K} S'$ is also a dualizing DG-module over $\til{A}\otimes_{\K} \til{B}$. 

By \cite[Theorem 7.10(2)]{Ye1}, there is some tilting module $L$ over $\til{A}$ such that $R \cong R'\otimes^{\mrm{L}}_{\til{A}} L$.
Similarly, there is a tilting DG-module $N$ over $\til{B}$, such that $S \cong S'\otimes^{\mrm{L}}_{\til{B}} N$ in $\cat{D}(\til{B})$.
Hence,
\[
R \otimes^{\mrm{L}}_{\K} S \cong R'\otimes^{\mrm{L}}_{\K} S' \otimes^{\mrm{L}}_{\til{A}\otimes_{\K} \til{B}} (L\otimes^{\mrm{L}}_{\K} N).
\]
Since by Lemma \ref{lem:box-tensor-of-tilting}, $L\otimes^{\mrm{L}}_{\K} N$ is a tilting DG-module over $\til{A}\otimes_{\K} \til{B}$, 
we deduce from \cite[Theorem 7.10(1)]{Ye1} that $R \otimes^{\mrm{L}}_{\K} S$
is a dualizing DG-module over $\til{A}\otimes_{\K} \til{B}$.
\end{proof}


\section{The twisted tensor product symmetric monoidal operation}

Let $\K$ be a Gorenstein noetherian ring of finite Krull dimension. 
We have seen in Corollary \ref{cor:finite} that for any $A \in \thecat$,
there exists a rigid dualizing DG-module over $A$ relative to $\K$.
Later in this section we will show that the rigid dualizing DG-module is unique,
up to isomorphism, so given $A \in \thecat$, 
let us denote by $R_A$ its rigid dualizing DG-module relative to $\K$,
and by $D_A$ the associated dualizing functor 
\[
D_A(-) := \mrm{R}\opn{Hom}_A(-,R_A).
\]

In \cite[Theorem 4.1]{AILN}, the main result of that paper, 
it is shown that if $A = A^0$ is a ring, under suitable finiteness conditions on $M,N$,
there is a functorial isomorphism
\begin{equation}\label{eqn:ailn-red}
\mrm{R}\opn{Hom}_{A\otimes^{\mrm{L}}_{\K} A}(A,M\otimes^{\mrm{L}}_{\K} N) \cong \mrm{R}\opn{Hom}_A(\mrm{R}\opn{Hom}_A(M,R_A),N).
\end{equation}
As we explained in \cite{Sh1}, 
at least when $\K$ is Gorenstein of finite Krull dimension, 
as assumed in this section,
the correct way to view this isomorphism is as follows.
As $N$ is assumed to have finitely generated cohomology,
there is an isomorphism 
\[
N \cong \mrm{R}\opn{Hom}_A(\mrm{R}\opn{Hom}_A(N,R),R),
\]
and plugging this into (\ref{eqn:ailn-red}), 
using adjunction, we may write its right hand side as
\begin{eqnarray}
\mrm{R}\opn{Hom}_A(\mrm{R}\opn{Hom}_A(M,R_A)\otimes^{\mrm{L}}_A \mrm{R}\opn{Hom}_A(N,R_A),R_A) = \nonumber\\
D_A(D_A(M)\otimes^{\mrm{L}}_A D_A(N)).\label{eqn:first-tensor-twist}
\end{eqnarray}
To understand this expression,
we now generalize to the DG-case the following construction from \cite{Sh1}:\\
Let $A,B \in \thecat$, 
and let
\[
F : \underbrace{\cat{D}(A)\times \cat{D}(A) \times \dots \times \cat{D}(A)}_n \to \cat{D}(B)
\]
be some functor. 
We define the twist of $F$ to be the functor 
\[
F^{!}:\underbrace{\cat{D}(A)\times \cat{D}(A) \times \dots \times \cat{D}(A)}_n \to \cat{D}(B)
\]
given by
\[
F^{!}(M_1,\dots,M_n) := D_B(F(D_A(M_1),\dots,D_A(M_n))).
\]
With this construction in hand, 
it is clear that (\ref{eqn:first-tensor-twist}) is simply the twist of the derived tensor product 
\[
-\otimes^{\mrm{L}}_A -: \cat{D}(A) \times \cat{D}(A) \to \cat{D}(A)
\]
Thus, we will denote it by $-\otimes^{!}_A-$, and refer to it as the twisted tensor product bifunctor.
Since $-\otimes^{\mrm{L}}_A-$ is a symmetric monoidal product on $\cat{D}^{-}_{\mrm{f}}(A)$,
and since 
\[
D_A : \cat{D}^{+}_{\mrm{f}}(A) \to \cat{D}^{-}_{\mrm{f}}(A)
\]
is an equivalence, it follows that $- \otimes^{!}_A -$ makes $\cat{D}^{+}_{\mrm{f}}(A)$ 
into a symmetric monoidal category. 
We will show in this section that (\ref{eqn:ailn-red}) generalizes to the DG-case. 
Our proof is based on a non-trivial adaptation to the commutative DG-case of \cite[Theorem 2.6]{Sh2} 
where we proved an analogue result for noncommutative algebras over a field which are finite over their center.

We begin by recalling a classical result, 
which is in some sense dual to the main result of this section.
\begin{prop}\label{prop:hochschild-homology-of-tensor}
Let $\K$ be a commutative ring, and let $A$ be K-flat  DG-ring over $\K$. For any $M,N \in \cat{D}(A)$. Then there is a natural isomorphism
\[
A \otimes^{\mrm{L}}_{A\otimes_{\K} A} (M\otimes^{\mrm{L}}_{\K} N) \cong M\otimes^{\mrm{L}}_A N
\]
in $\cat{D}(A)$.
\end{prop}
\begin{proof}
If $A=A^0$ is a ring, this is a well known classical result, and the same proof works in this more general situation. See \cite[Remark 3.11]{AILN}.
\end{proof}
\begin{prop}\label{prop:hom-of-dualizing}
Let $f:A \to B$ be a homomorphism of DG rings, with $A$ being cohomologically noetherian, and let $R$ be a dualizing DG-module over $A$. Let $M \in \cat{D}(B)$, such that $\opn{For}_f(M) \in \cat{D}_{\mrm{f}}(A)$, and let $N\in \cat{D}_{\mrm{f}}(A)$.  Then there is a natural isomorphism
\[
\mrm{R}\opn{Hom}_A(D(M),D(N)) \cong \mrm{R}\opn{Hom}_A(N,M)
\]
in $\cat{D}(B)$, where $D(-) := \mrm{R}\opn{Hom}_A(-,R)$.
\end{prop}
\begin{proof}
Let $R\iso I$ be a K-injective resolution over $A$, and let $P \iso N$ be a K-projective resolution over $A$.
Note that the DG-module $\opn{Hom}_A(P,I)$ is also K-injective. Hence, we have that
\[
\mrm{R}\opn{Hom}_A(D(M),D(N)) \cong \opn{Hom}_A(\opn{Hom}_A(M,I),\opn{Hom}_A(P,I)).
\]
Using adjunction twice, there are natural isomorphisms in $\cat{D}(B)$:
\begin{eqnarray}
\opn{Hom}_A(\opn{Hom}_A(M,I),\opn{Hom}_A(P,I)) \cong\nonumber\\
\opn{Hom}_A(\opn{Hom}_A(M,I)\otimes_A P,I)\cong\nonumber\\
\opn{Hom}_A(P,\opn{Hom}_A(\opn{Hom}_A(M,I),I)).\nonumber
\end{eqnarray}
By \cite[Proposition 7.2(2)]{Ye1}, the natural $B$-linear map $M\to \opn{Hom}_A(\opn{Hom}_A(M,I),I)$ is a quasi-isomorphism, so the result follows.
\end{proof}

Here is the main result of this section, a generalization of \cite[Theorem 4.1]{AILN} to the DG-case.

\begin{thm}\label{thm:reduction}
Let $\K$ be a Gorenstein noetherian ring of finite Krull dimension,
let $A \in \thecat$, and let $(R_A,\rho)$ be a rigid dualizing DG-module over $A$ relative to $\K$.
Then there is a functorial isomorphism
\begin{eqnarray}
\mrm{R}\opn{Hom}_{A\otimes^{\mrm{L}}_{\K} A}(A,M\otimes^{\mrm{L}}_{\K} N) \cong \nonumber\\
\mrm{R}\opn{Hom}_A(\mrm{R}\opn{Hom}_A(M,R_A)\otimes^{\mrm{L}}_A \mrm{R}\opn{Hom}_A(N,R_A),R_A) \nonumber
\end{eqnarray}
for any $M \in \cat{D}^{b}_{\mrm{f}}(A)$
such that $\mrm{R}\opn{Hom}_A(M,R_A)$ has finite flat dimension over $\K$,
and for any $N \in \cat{D}^{-}_{\mrm{f}}(A)$.
\end{thm}
\begin{proof}
Step 1:
Let $f:A \hequiv \til{A}$ be the homological equivalence constructed in Proposition \ref{prop:res},
and let $R_{\til{A}} := \eqfun_{f}(R_A)$. 
By Proposition \ref{prop:transfer}, 
$R_{\til{A}}$ has the structure of a rigid dualizing DG-module over $\til{A}$ relative to $\K$.
Also by Proposition \ref{prop:transfer}, 
there are functorial isomorphisms
\[
\mrm{R}\opn{Hom}_{A\otimes^{\mrm{L}}_{\K} A}(A,M\otimes^{\mrm{L}}_{\K} N) 
\cong 
\eqfun_{f^{-1}}(\mrm{R}\opn{Hom}_{\til{A}\otimes^{\mrm{L}}_{\K} \til{A}}(\til{A},\eqfun_f(M)\otimes^{\mrm{L}}_{\K} \eqfun_f(N)))
\]
and
\[
M\otimes^{!}_A N \cong
\eqfun_{f^{-1}} ( \eqfun_f(M)\otimes_{\til{A}}^{!} \eqfun_f(N) ),
\]
(where the $-\otimes^{!}-$ are defined with respect to $R_A$ and $R_{\til{A}}$ respectively).
Hence, it is enough to prove the result for $A = \til{A}$.

Step 2: By step 1, we have that $A$ is K-flat over $\K$, 
so by \cite[Theorem 6.15]{Ye2}, we have that
\[
\mrm{R}\opn{Hom}_{A\otimes^{\mrm{L}}_{\K} A}(A,M\otimes^{\mrm{L}}_{\K} N) \cong
\mrm{R}\opn{Hom}_{A\otimes_{\K} A}(A,M\otimes^{\mrm{L}}_{\K} N).
\]
By Theorem \ref{thm:tensor-of-dualizing},
the DG-module 
\[
R_A\otimes^{\mrm{L}}_{\K} R_A
\]
is a dualizing DG-module over $A\otimes_{\K} A$. 
It is clear that $M\otimes^{\mrm{L}}_{\K} N$ and $A$ both belong to $\cat{D}_{\mrm{f}}(A\otimes_{\K} A)$,
so by Proposition \ref{prop:hom-of-dualizing}, 
there is a functorial isomorphism
\begin{eqnarray}
\mrm{R}\opn{Hom}_{A\otimes_{\K} A}(A,M\otimes^{\mrm{L}}_{\K} N) \cong\nonumber\\
\mrm{R}\opn{Hom}_{A\otimes_{\K} A}(D_{A\otimes_{\K} A}(M\otimes^{\mrm{L}}_{\K} N),D_{A\otimes_{\K} A}(A))\label{reduction-after-dualizing},
\end{eqnarray}
in $\cat{D}(A)$,
where we have set
\[
D_{A\otimes_{\K} A}(-) := \mrm{R}\opn{Hom}_{A\otimes_{\K} A}(-,R_A\otimes^{\mrm{L}}_{\K} R_A).
\]
By rigidity of $R_A$, 
we know that there is an isomorphism
\[
\rho_A : R_A \to D_{A\otimes_{\K} A}(A)
\]
in $\cat{D}(A)$, and using this isomorphism, (\ref{reduction-after-dualizing}) is naturally isomorphic to
\[
\mrm{R}\opn{Hom}_{A\otimes_{\K} A}(\mrm{R}\opn{Hom}_{A\otimes_{\K} A}(M\otimes^{\mrm{L}}_{\K} N,R_A\otimes^{\mrm{L}}_{\K} R_A),R_A).
\]
By Proposition \ref{prop:ffd}, 
$R_A$ has finite flat dimension over $\K$.
Hence, by Lemma \ref{lem:tensor-split},
there is a natural isomorphism
\[
\mrm{R}\opn{Hom}_{A}(M,R_A) \otimes^{\mrm{L}}_{\K} \mrm{R}\opn{Hom}_{A}(N,R_A)
\cong
\mrm{R}\opn{Hom}_{A\otimes_{\K} A}(M\otimes^{\mrm{L}}_{\K} N,R_A\otimes^{\mrm{L}}_{\K} R_A)
\]
in $\cat{D}(A\otimes_{\K} A)$.
Hence, we have that
\begin{eqnarray}
\mrm{R}\opn{Hom}_{A\otimes_{\K} A}(\mrm{R}\opn{Hom}_{A\otimes_{\K} A}(M\otimes^{\mrm{L}}_{\K} N,R_A\otimes^{\mrm{L}}_{\K} R_A),R_A) \cong \nonumber\\
\mrm{R}\opn{Hom}_{A\otimes_{\K} A}(
\mrm{R}\opn{Hom}_{A}(M,R_A) \otimes^{\mrm{L}}_{\K} \mrm{R}\opn{Hom}_{A}(N,R_A) ,R_A)\cong\nonumber\\
\mrm{R}\opn{Hom}_A(A\otimes^{\mrm{L}}_{A\otimes_{\K} A}(\mrm{R}\opn{Hom}_{A}(M,R_A) \otimes^{\mrm{L}}_{\K} \mrm{R}\opn{Hom}_{A}(N,R_A)),R_A)\nonumber,
\end{eqnarray}
where the second isomorphism is adjunction.
By Proposition \ref{prop:hochschild-homology-of-tensor}, 
we get an isomorphism of functors
\begin{eqnarray}
\mrm{R}\opn{Hom}_A(A\otimes^{\mrm{L}}_{A\otimes_{\K} A}(\mrm{R}\opn{Hom}_{A}(M,R_A) \otimes^{\mrm{L}}_{\K} \mrm{R}\opn{Hom}_{A}(N,R_A)),R_A) \cong\nonumber\\
\mrm{R}\opn{Hom}_A(\mrm{R}\opn{Hom}_{A}(M,R_A)\otimes^{\mrm{L}}_A \mrm{R}\opn{Hom}_{A}(N,R_A), R_A),\nonumber
\end{eqnarray}
and this completes the proof.
\end{proof}

\begin{cor}\label{cor:group}
Let $\K$ be a Gorenstein noetherian ring of finite Krull dimension.
For any $A \in \thecat$, 
denote by $\mathcal{D}_A$ the set of isomorphism classes of dualizing DG-modules over $A$.
Then the operation 
\[
\mrm{R}\opn{Hom}_{A\otimes^{\mrm{L}}_{\K} A}(A,-\otimes^{\mrm{L}}_{\K}-)
\]
defines a group structure on $\mathcal{D}_A$, 
and any rigid dualizing DG-module $(R_A,\rho)$ is a unit of this group.
In particular, the rigid dualizing DG-module is unique up to isomorphism.
\end{cor}
\begin{proof}
Let $(R_A,\rho)$ be some rigid dualizing DG-module over $A$ relative to $\K$ (such exists by Corollary \ref{cor:finite}).
Note that for any $R \in \mathcal{D}_A$, 
by \cite[Theorem 7.10(2)]{Ye1}
\[
P := \mrm{R}\opn{Hom}_A(R,R_A)
\]
is a tilting DG-module over $A$. 
By \cite[Theorem 6.5]{Ye1}, 
this implies that $P$ is a perfect DG-module over $A$,
and by \cite[Theorem 5.11]{Ye1}, 
this implies that $P$ has finite flat dimension relative to $\cat{D}(A)$,
so that $P$ has finite flat dimension over $\K$.
It follows that Theorem \ref{thm:reduction} applies to any pair $R_1,R_2 \in \mathcal{D}_A$,
so it is enough to show that $-\otimes^{!}_A-$ defines a group structure on $\mathcal{D}_A$.
But this follows immediately from the fact that the set of isomorphism classes of tilting DG-modules form a group
with respect to $-\otimes^{\mrm{L}}_A -$.
Finally, it is clear from Theorem \ref{thm:reduction} that $R_A$ is a unit of this group,
so the fact that the rigid dualizing DG-module is unique up to isomorphism follows from the uniqueness of the identity element in a group.
\end{proof}

\begin{rem}
For a DG-ring $A$, we denote by $\opn{DPic}(A)$ its derived Picard group. Its objects are isomorphism classes of  tilting DG-modules, and the group operation is given by the derived tensor product $-\otimes^{\mrm{L}}_A -$. 
Then it is clear from the above proof that the map $\mrm{R}\opn{Hom}_A(-,R_A)$ defines a group isomorphism between $\mathcal{D}_A$, and $\opn{DPic}(A)$.
\end{rem}

\begin{notation}
In the rest of the paper, 
given a Gorenstein noetherian ring $\K$ of finite Krull dimension, 
and $A \in \thecat$, we will denote by $R_A$ the unique rigid dualizing DG-module over $A$ relative to $\K$.
\end{notation}

\begin{rem}
We could have given an easier proof of the fact that the rigid dualizing DG-module is unique,
imitating \cite[Theorem 3.3]{YZ2}. 
However, we find the above proof more conceptual: 
the rigid dualizing DG-module is unique up to isomorphism,
because it is the identity element of a canonical group whose operation is the functor underlying derived Hochschild cohomology. See also \cite[Theorem 9.7]{Ye1} for a similar uniqueness result in the DG-setting (but with slightly different finiteness assumptions).
\end{rem}

We may now obtain a slight improvement of Theorem \ref{thm:reduction}.

\begin{prop}\label{prop:ffd-of-dual}
Let $\K$ be a Gorenstein noetherian ring of finite Krull dimension.
Given $M \in \cat{D}^{\mrm{b}}_{\mrm{f}}(A)$, 
if $M$ has finite flat dimension over $\K$, 
then
\[
\mrm{R}\opn{Hom}_A(M,R_A)
\]
also has finite flat dimension over $\K$.
Hence, in Theorem \ref{thm:reduction}, 
it is enough to assume that $M$ itself has finite flat dimension over $\K$.
\end{prop}
\begin{proof}
Using Proposition \ref{prop:transfer}, 
it is clear that we may replace $A$ by the homologically equivalent $\til{A}$ from Proposition \ref{prop:res},
so we may assume that $A^0 \in \thecat$. 
But then, by Corollary \ref{cor:finite},
\[
\mrm{R}\opn{Hom}_{A^0}(A,R_{A^0})
\]
has the structure of a rigid dualizing DG-module over $A$ relative to $\K$, 
so by uniqueness of rigid dualizing DG-modules, we must have an isomorphism
\[
R_A \cong \mrm{R}\opn{Hom}_{A^0}(A,R_{A^0})
\]
Now, using adjunction, 
there is an isomorphism
\[
\mrm{R}\opn{Hom}_A(M,R_A) \cong \mrm{R}\opn{Hom}_{A^0}(M,R_{A^0}),
\]
so it is enough to show that the latter has finite flat dimension over $\K$,
and this follows from \cite[Theorem 1.2(1)]{AILN}.
\end{proof}

\begin{rem}
As $-\otimes^{!}_A -$ makes $\cat{D}^{+}_{\mrm{f}}(A)$ into a symmetric monoidal category,
it is tempting to ask if one can extend the isomorphism of Theorem \ref{thm:reduction} to any $M,N \in \cat{D}^{+}_{\mrm{f}}(A)$. 
Unfortunately, the answer to this is negative, for trivial reasons:
Let $\K$ be a Gorenstein noetherian ring of finite Krull dimension which is not a regular ring,
and let $A = \K$. 
Since $A$ has infinite global dimension,
one may find finitely generated $A$-modules $M,N$,
such that 
\[
\opn{Tor}^A_n(M,N) \ne 0
\]
for infinitely many $n$. 
It follows that
\[
\mrm{R}\opn{Hom}_{A\otimes^{\mrm{L}}_{\K} A}(A,M\otimes^{\mrm{L}}_{\K} N) = M\otimes^{\mrm{L}}_{\K} N \notin \cat{D}^{+}_{\mrm{f}}(A),
\]
so in particular
\[
\mrm{R}\opn{Hom}_{A\otimes^{\mrm{L}}_{\K} A}(A,M\otimes^{\mrm{L}}_{\K} N) \ncong M\otimes^{!}_A N.
\]
\end{rem}

\section{Rigid dualizing DG-modules and cohomologically essentially smooth maps}

Recall that a map $f:A \to B$ between two noetherian rings is called essentially smooth,
if $f$ is formally smooth and essentially of finite type. 
This implies that $f$ is flat, 
and that $\Omega^1_{B/A}$ is a finitely generated projective module.
Decomposing $\opn{Spec} B = \sqcup_{i=1}^n \opn{Spec} B_i$ to its connected components,
the induced map $A \to B_i$ is also essentially smooth,
and $\Omega^1_{B_i/A}$ is a projective module of fixed rank $n_i$.
We set 
\begin{equation}\label{eqn:dfn-of-omega}
\Omega_{B/A} := \bigoplus_{i=1}^n \Omega^{n_i}_{B_i/A}[n_i].
\end{equation}
This is a tilting complex over $B$, 
and there is an isomorphism 
\[
f^{!}(-) \cong \Omega_{B/A} \otimes_A -
\]
of functors $\cat{D}^{+}_{\mrm{f}}(A) \to \cat{D}^{+}_{\mrm{f}}(B)$. 
The goal of this section is to generalize these facts to DG-rings. 

An essentially smooth map $A \to B$ between noetherian rings is always flat, 
so that $B$ has flat dimension $0$ as an object of $\cat{D}(A)$. Similarly, we define

\begin{dfn}
Given a map $A \to B$ of DG-rings, 
we say it has flat dimension $0$ if $B$ has $0$ flat dimension relative to $\cat{D}(A)$. 
Explicitly, this means that given a DG $A$-module $M$, 
such that the diameter of $\mrm{H}(M)$ is $d$,
then the diameter of $\mrm{H}(M\otimes^{\mrm{L}}_A B)$ is smaller or equal to $d$.
\end{dfn}

Here is an important implication of this definition.

\begin{lem}\label{lem:fd0iso}
Let $A\to B$ be a DG-ring map of flat dimension $0$.
Then there is an isomorphism 
\[
\bar{A}\otimes^{\mrm{L}}_A B \cong \bar{B}.
\]
\end{lem}
\begin{proof}
Note that $\mrm{H}^0(\bar{A}\otimes^{\mrm{L}}_A B) \ne 0$, 
so the flat dimension assumption implies that
\[
\bar{A}\otimes^{\mrm{L}}_A B \cong H^0(\bar{A}\otimes^{\mrm{L}}_A B).
\]
Since $B$ is non-positive, 
the latter is isomorphic by the Kunneth trick to
\[
\bar{A}\otimes_A \mrm{H}^0(B) = \bar{B}.
\]
\end{proof}

With this definition in hand, we arrive to the following definition of smoothness. 
See \cite[4.5.1]{Ga} for an equivalent definition in the finite type case,
and \cite[Section 2.2.2]{TV} for a detailed discussion.

\begin{dfn}
Let $f:A \to B$ be a DG-ring map between two cohomologically noetherian DG-rings.
We say that $f$ is cohomologically essentially smooth if
it has flat dimension $0$, 
and the induced map $\bar{f}:\bar{A}\to\bar{B}$ is essentially smooth.
\end{dfn}

\begin{rem}
To see why this is a good definition, 
note that if $A\to B$ is smooth in any reasonable sense, 
it must remain so after applying base change with respect to $A \to \bar{A}$,
and now the isomorphism of Lemma \ref{lem:fd0iso} which follows from the flat dimension $0$ assumption, implies that $\bar{A} \to \bar{B}$ is smooth.
\end{rem}

One of the most important properties of essentially smooth maps between noetherian rings is that they are Gorenstein.
We will now define a notion of a Gorenstein map in the DG-setting.
See \cite[Definition 7.3.2]{Ga} for an equivalent definition.

\begin{dfn}
Let $f:A \to B$ be a map between two cohomologically noetherian DG-rings, 
and assume that $A$ has dualizing DG-modules.
We say that $f$ is Gorenstein if for any dualizing DG-module $R$ over $A$, 
the DG-module $R\otimes^{\mrm{L}}_A B$ is a dualizing DG-module over $B$.
\end{dfn}

The next proposition is a DG generalization of \cite[Proposition II.5.14]{RD}.

\begin{prop}\label{prop:tensor-with-ffd}
Let $A$ be a cohomologically noetherian DG-ring, and let $M \in \cat{D}^{-}_{\mrm{f}}(A)$, $N \in \cat{D}^{+}(A)$, and $K \in \cat{D}^{\mrm{b}}(A)$, such that $K$ has finite flat dimension relative to $\cat{D}(A)$.
Then there is an isomorphism
\[
\mrm{R}\opn{Hom}_A(M,N) \otimes^{\mrm{L}}_A K \iso \mrm{R}\opn{Hom}_A(M,N\otimes^{\mrm{L}}_A K)
\]
in $\cat{D}(A)$, functorial in $M,N,K$.
\end{prop}
\begin{proof}
Fixing $N,K$ as above, because of the boundedness assumptions on them, 
it is clear that both of the functors
$\mrm{R}\opn{Hom}_A(-,N) \otimes^{\mrm{L}}_A K$ and 
$\mrm{R}\opn{Hom}_A(-,N\otimes^{\mrm{L}}_A K)$
have bounded above cohomological displacement, in the sense of \cite[Definition 2.1(3)]{Ye1}.
Denoting by $\eta_M$ the natural morphism
\[
\eta_M : \mrm{R}\opn{Hom}_A(M,N) \otimes^{\mrm{L}}_A K \to \mrm{R}\opn{Hom}_A(M,N\otimes^{\mrm{L}}_A K)
\]
in $\cat{D}(A)$, 
obtained by replacing $M$ by a K-projective resolution,
and $K$ by a K-flat resolution, it is clear that $\eta_A$ is an isomorphism.
Hence, the result follows from \cite[Theorem 2.11(1)]{Ye1}.
\end{proof}

We shall need the next result which is due to Lurie.

\begin{prop}\label{prop:lurie-reduction}
Let $A$ be a cohomologically noetherian DG-ring, 
let $R \in \cat{D}^{\mrm{b}}_{\mrm{f}}(A)$,
and assume that 
\[
\mrm{R}\opn{Hom}_A(\bar{A},R)
\]
is a dualizing complex over $\bar{A}$.
Then $R$ is a dualizing DG-module over $A$.
\end{prop}
\begin{proof}
This is a a particular case of \cite[Proposition 4.3.8]{Lu}.
\end{proof}

Here is the first main result of this section. 
See \cite[Corollary 7.3.6]{Ga} for a similar result (with a very different proof).

\begin{thm}\label{thm:gor-lifting}
Let $f:A \to B$ be a map between cohomologically noetherian DG-rings which is of flat dimension $0$,
and assume that $A$ has dualizing DG-modules, and that the induced map $\bar{f}:\bar{A}\to\bar{B}$ is Gorenstein.
Then $f$ is also Gorenstein.
\end{thm}
\begin{proof}
Let $\varphi:\til{B} \iso B$ be a K-flat resolution of $f$.
Since $\bar{A}$ is a ring which is a finite $\mrm{H}(A)$-algebra, 
it follows by \cite[Proposition 1.7(3)]{YZ1}, 
that there exists a (multiplicative!) pseudo-finite semi-free resolution 
$\psi: \til{A} \iso \bar{A}$ of the canonical map $A \to \bar{A}$.
In terms of these resolutions, 
the $0$-flat dimension assumption on $f$ implies by Lemma \ref{lem:fd0iso},
that there is a quasi-isomorphism
$\chi:\til{A}\otimes_A \til{B} \cong \bar{B}$.
Note further that $\til{A}$ is K-projective over $A$, 
and $\til{A}\otimes_A \til{B}$ is K-projective over $\til{B}$.

Let $R$ be a dualizing DG-module over $A$.
Because of cohomological finiteness of $A \to \bar{A}$,
\[
\mrm{R}\opn{Hom}_A(\bar{A},R)
\]
is a dualizing complex over the ring $\bar{A}$.
Since
\[
\opn{For}_{\psi}(\mrm{R}\opn{Hom}_A(\bar{A},R)) = \mrm{R}\opn{Hom}_A(\til{A},R),
\]
by Proposition \ref{prop:tensor-with-forget},
there is an isomorphism
\begin{equation}\label{eqn:til-to-bar}
\mrm{R}\opn{Hom}_A(\til{A},R) \otimes^{\mrm{L}}_{\til{A}} \bar{A} \cong \mrm{R}\opn{Hom}_A(\bar{A},R).
\end{equation}
The Gorenstein assumption on $\bar{f}$ says that
\[
\mrm{R}\opn{Hom}_A(\bar{A},R)\otimes^{\mrm{L}}_{\bar{A}} \bar{B}
\]
is a dualizing complex over $\bar{B}$, and using (\ref{eqn:til-to-bar}),
we see that
\[
\mrm{R}\opn{Hom}_A(\bar{A},R)\otimes^{\mrm{L}}_{\bar{A}} \bar{B} \cong \mrm{R}\opn{Hom}_A(\til{A},R)\otimes^{\mrm{L}}_{\til{A}} \bar{B}
\]
so that the latter is also a dualizing complex over $\bar{B}$. 
Hence,
\[
\opn{For}_{\chi}(\mrm{R}\opn{Hom}_A(\til{A},R)\otimes^{\mrm{L}}_{\til{A}} \bar{B}) \cong
\opn{Hom}_A(\til{A},R) \otimes_A \til{B}
\]
is a dualizing DG-module over $\til{A}\otimes_A \til{B}$.

Because $\til{A}$ is pseudo-finite semi-free over $A$, 
and since by assumption $B$ (and hence $\til{B}$) has finite flat dimension relative to $\cat{D}(A)$,
by the proof of Proposition \ref{prop:tensor-with-ffd}, 
the natural map
\[
\opn{Hom}_A(\til{A},R) \otimes_A \til{B} \to \opn{Hom}_A(\til{A},R \otimes_A \til{B})
\]
which is $\til{A}\otimes_A \til{B}$-linear, is a quasi-isomorphism.
By adjunction, there is an isomorphism
\[
\opn{Hom}_A(\til{A},R\otimes_A \til{B}) \cong
\opn{Hom}_{\til{B}}(\til{A}\otimes_A \til{B},R\otimes_A \til{B}) 
\]
of DG $\til{A}\otimes_A \til{B}$-modules, 
so we deduce that
\[
\opn{Hom}_{\til{B}}(\til{A}\otimes_A \til{B},R\otimes_A \til{B}) 
\]
is also a dualizing DG-module over $\til{A}\otimes_{A} \til{B}$.
It follows by \cite[Proposition 7.5(1)]{Ye1} that
\[
\mrm{R}\opn{Hom}_{\til{A}\otimes_A \til{B}}(\bar{B} , 
\mrm{R}\opn{Hom}_{\til{B}}(\til{A}\otimes_A \til{B},R\otimes_A \til{B}) ) \cong
\mrm{R}\opn{Hom}_{\til{B}}(\bar{B},R\otimes_A \til{B})
\]
is a dualizing complex over $\bar{B}$.

Finally note that $R\otimes^{\mrm{L}}_A \til{B} \in \cat{D}^{\mrm{b}}_{\mrm{f}}(\til{B})$, 
because $B$ is assumed to have finite flat dimension relative to $\cat{D}(A)$.
Hence, by Proposition \ref{prop:lurie-reduction},
$R\otimes_A \til{B}$ is a dualizing DG-module over $\til{B}$, 
so that $R\otimes^{\mrm{L}}_A B$ is a dualizing DG-module over $B$.
Hence $f$ is Gorenstein.
\end{proof}

Recall that if $A$ is any cohomologically noetherian DG-ring, 
then by \cite[Theorem 6.11]{Ye1}, 
the map $-\otimes^{\mrm{L}}_A \bar{A}$ induces a bijection between isomorphism classes of tilting DG-modules over $A$, 
and tilting complexes over $\bar{A}$.
Given a cohomologically essentially smooth map $f:A \to B$ between cohomologically noetherian DG-rings,
the tilting $\bar{B}$-complex $\Omega_{\bar{B}/\bar{A}}$ associated to the essentially smooth map $\bar{f}:\bar{A}\to \bar{B}$ 
was defined in (\ref{eqn:dfn-of-omega}).
For such a map $f$, we define $\Omega_{B/A}$ to be the unique tilting DG $B$-module, 
such that $\Omega_{B/A} \otimes^{\mrm{L}}_B \bar{B} \cong \Omega_{\bar{B}/\bar{A}}$.

\begin{cor}\label{cor:dual-smooth}
Let $A \to B$ be a cohomologically essentially smooth map between cohomologically noetherian DG-rings,
and let $R$ be a dualizing DG-module over $A$. 
Then $R\otimes^{\mrm{L}}_A \Omega_{B/A}$ is a dualizing DG-module over $B$.
\end{cor}
\begin{proof}
A cohomologically essentially smooth map satisfies the assumptions of Theorem \ref{thm:gor-lifting}, 
so it is Gorenstein. 
Hence, $R\otimes^{\mrm{L}}_A B$ is a dualizing DG-module over $B$, 
and since $\Omega_{B/A}$ is a tilting DG $B$-module, 
the result follows from \cite[Theorem 7.10(1)]{Ye1}.
\end{proof}

\begin{exa}\label{ex:qis-is-smooth}
If $f:A\to B$ is a quasi-isomorphism between cohomologically noetherian DG-ring,
then $f$ is cohomologically essentially smooth.
Since $\Omega_{\bar{B}/\bar{A}} = \bar{B}$, 
we have that $\Omega_{B/A} = B$.
\end{exa}

It is clear that the composition of two cohomologically essentially smooth maps is again cohomologically essentially smooth.
As for the differentials, we have, as in the classical case:

\begin{prop}\label{prop:compose-smooth}
Let $A \to B$ and $B \to C$ be two cohomologically essentially smooth maps between cohomologically noetherian DG-rings. Then the composed map $A \to C$ is also cohomologically essentially smooth, 
and there is an isomorphism
\[
\Omega_{B/A} \otimes^{\mrm{L}}_B \Omega_{C/B} \cong \Omega_{C/A}
\]
in $\cat{D}(C)$.
\end{prop}
\begin{proof}
The fact that $A \to C$ is cohomologically essentially smooth is clear from the definitions.
For the second claim, by the definition of $\Omega$, 
it is enough to show that these are isomorphic after applying $-\otimes^{\mrm{L}}_C \bar{C}$.
We have that
\begin{eqnarray}
(\Omega_{B/A} \otimes^{\mrm{L}}_B \Omega_{C/B} )\otimes^{\mrm{L}}_C \bar{C} \cong\nonumber\\
\Omega_{B/A} \otimes^{\mrm{L}}_B \Omega_{\bar{C}/\bar{B}} \cong\nonumber\\
\Omega_{B/A} \otimes^{\mrm{L}}_B \bar{B}\otimes^{\mrm{L}}_{\bar{B}} \bar{C} \otimes^{\mrm{L}}_{\bar{C}} \Omega_{\bar{C}/\bar{B}}\cong\nonumber\\
\Omega_{\bar{B}/\bar{A}} \otimes^{\mrm{L}}_{\bar{B}} \Omega_{\bar{C}/\bar{B}} \cong \Omega_{\bar{C}/\bar{A}},\nonumber
\end{eqnarray}
where the last isomorphism follows from \cite[Proposition 3.4]{YZ1}. This proves the result.
\end{proof}

\begin{cor}\label{cor:omega-qis}
Let $\til{A} \iso A$ and $\til{B} \to B$ be quasi-isomorphisms between cohomologically noetherian DG-rings, 
and let $A \to \til{B}$ be a cohomologically essentially smooth map. 
Then are isomorphisms
\[
\Omega_{B/A} \cong \Omega_{B/\til{A}}
\]
and
\[
\Omega_{\til{B}/A} \otimes^{\mrm{L}}_{\til{B}} B \cong \Omega_{B/A}
\]
in $\cat{D}(B)$
\end{cor}
\begin{proof}
Combine Example \ref{ex:qis-is-smooth} with Proposition \ref{prop:compose-smooth}.
\end{proof}

\begin{prop}\label{prop:smooth-base-change}
Let $A \to B$ be a cohomologically essentially smooth map  between cohomologically noetherian DG-rings,
and let $A \to C$ be a K-flat DG-ring map, such that $C$ and $B\otimes_A C$ are also cohomologically noetherian.
Then the induced map $C \to B\otimes_A C$ is also cohomologically essentially smooth,
and there is an isomorphism
\[
\Omega_{B/A} \otimes_A C \cong \Omega_{B\otimes_A C/C}
\]
in $\cat{D}(B\otimes_A C)$.
\end{prop}
\begin{proof}
Let $M \in \cat{D}(C)$,
and assume that the diameter of $\mrm{H}(M)$ is $d$.
Since 
\[
M\otimes^{\mrm{L}}_C (C\otimes_A B) \cong M\otimes^{\mrm{L}}_A B,
\]
we see that the diameter of $\mrm{H}(M\otimes^{\mrm{L}}_C (C\otimes_A B))$ is $\le d$, 
so that $B\otimes_A C$ has flat dimension $0$ over $C$.
Using Lemma \ref{lem:fd0iso} twice, 
we have that
\begin{equation}\label{eqn:smooth-h0-base-change}
\mrm{H}^0(B\otimes_A C) \cong \bar{C}\otimes^{\mrm{L}}_C (B\otimes_A C) \cong \bar{C}\otimes^{\mrm{L}}_A B \cong \bar{C}\otimes^{\mrm{L}}_{\bar{A}} \bar{A}\otimes^{\mrm{L}}_A B = \bar{C}\otimes_{\bar{A}} \bar{B}.
\end{equation}
Since $\bar{A} \to \bar{B}$ is essentially smooth,
by base change $\bar{C} \to \bar{C}\otimes_{\bar{A}} \bar{B}$ is essentially smooth,
so we deduce that $C \to B\otimes_A C$ is  cohomologically essentially smooth.
To show the second claim, 
note first that by Lemma \ref{lem:box-tensor-of-tilting},
$\Omega_{B/A} \otimes_A C$ is a tilting $B\otimes_A C$ DG-module.
Thus, in order to show that the two tilting DG-modules 
$\Omega_{B/A} \otimes_A C $ and $\Omega_{B\otimes_A C/C}$ are isomorphic, 
it is enough to show that 
\begin{equation}\label{eqn:smooth-base-change-omega}
(\Omega_{B/A} \otimes_A C) \otimes^{\mrm{L}}_{B\otimes_A C} \mrm{H}^0(B\otimes_A C) \cong \Omega_{B\otimes_A C/C} \otimes^{\mrm{L}}_{B\otimes_A C} \mrm{H}^0(B\otimes_A C).
\end{equation}
The right hand side of (\ref{eqn:smooth-base-change-omega}) satisfies, 
by definition of $\Omega$ and by its base change property over commutative rings:
\[
\Omega_{B\otimes_A C/C} \otimes^{\mrm{L}}_{B\otimes_A C} \mrm{H}^0(B\otimes_A C) \cong
\Omega_{\mrm{H}^0(B\otimes_A C)/\bar{C}} \cong \Omega_{\bar{B}\otimes_{\bar{A}} \bar{C}/\bar{C}}
\cong \Omega_{\bar{B}/\bar{A}}\otimes_{\bar{A}} \bar{C}.
\]
On the other hand, 
and using the sequence of isomorphisms (\ref{eqn:smooth-h0-base-change}),
the left hand side of (\ref{eqn:smooth-base-change-omega}) satisfies:
\begin{eqnarray}
(\Omega_{B/A} \otimes_A C) \otimes^{\mrm{L}}_{B\otimes_A C} \mrm{H}^0(B\otimes_A C) \cong \nonumber\\
(\Omega_{B/A} \otimes_A C) \otimes^{\mrm{L}}_{B\otimes_A C} (B\otimes^{\mrm{L}}_A \bar{C}) \cong \nonumber\\
\Omega_{B/A} \otimes_A \bar{C} \cong \Omega_{B/A} \otimes^{\mrm{L}}_B B \otimes_A \bar{C} \cong \nonumber\\
\Omega_{B/A} \otimes^{\mrm{L}}_B (\bar{B}\otimes_{\bar{A}} \bar{C}) \cong \Omega_{\bar{B}/\bar{A}}\otimes_{\bar{A}} \bar{C},\nonumber
\end{eqnarray}
and this completes the proof.
\end{proof}

\begin{lem}\label{lem:hequiv-of-map}
Let $\K$ be a noetherian ring, 
and let $f:A \to B$ be a map in $\thecat$.
Then there are homological equivalences $\varphi: A \hequiv \til{A}$
and $\psi:B \hequiv \til{B}$ over $\K$,
and a map $\til{f}:\til{A} \to \til{B}$ in $\thecat$,
such that $f$ and $\til{f}$ are homologically equivalent in the sense of Definition \ref{dfn:equiv-of-morphisms},
and moreover, we have that $\til{A}^0, \til{B}^0 \in \thecat$,
and the map $\til{f}^0:\til{A}^0 \to \til{B}^0$ is a $\K$-algebra map.
\end{lem}
\begin{proof}
By Lemma \ref{lem:var-res}, there is a commutative diagram
\[
\xymatrix{
A \ar[r]\ar[d] & A_{\opn{loc}} \ar[d] & \til{A}\ar[l] & \\
B \ar[r] & B_{\opn{loc}}
}
\]
in which all horizontal maps are quasi-isomorphisms,
and $\til{A}^0$ is a localization of a polynomial ring over $\K$. 
In particular, $\til{A}^0 \in \thecat$.

We will now make a process similiar to the construction in \cite[Lemma 7.8]{Ye1}, 
but in a slightly more general setting.
Let $\pi_{B_{\opn{loc}}} : B_{\opn{loc}} \to \bar{B}$ be the canonical map, 
and let $T$ be the set of elements $t$ in $B_{\opn{loc}}^0$ such that $\pi_{B_{\opn{loc}}}(t)$ is invertible in $\bar{B}$.
Set 
\[
B_{\opn{loc}}' := B_{\opn{loc}} \otimes_{B^0_{\opn{loc}}} T^{-1} B^0_{\opn{loc}}.
\]
The map $B_{\opn{loc}} \to B_{\opn{loc}}'$ induced by localization is clearly a quasi-isomorphism.
Since $\til{A}^0 \to \bar{B}$ is essentially of finite type, there is a map
\[
\varphi: \til{A}^0[t_1,\dots,t_n] \to \bar{B}
\]
which is surjective after localization. 
Let $U$ be the set of elements $u$ in $\til{A}^0[t_1,\dots,t_n]$,
such that $\varphi(u)$ is invertible in $\bar{B}$.
By construction, every preimage of a unit in $\bar{B}$ is a unit in $(B_{\opn{loc}}')^0$.
Hence, the map $\til{A} \to B_{\opn{loc}}'$ obtained from composing 
\[
\til{A} \to A_{\opn{loc}} \to B_{\opn{loc}} \to B_{\opn{loc}}'
\]
factors through 
\[
\til{A} \to \til{A}\otimes_{\til{A}^0} U^{-1}\til{A}^0[t_1,\dots,t_n] \to B_{\opn{loc}}',
\]
and the composed map
\[
\til{A}\otimes_{\til{A}^0} U^{-1}\til{A}^0[t_1,\dots,t_n] \to B_{\opn{loc}}' \to \bar{B}
\]
is surjective. 
Since $\mrm{H}^n(B_{\opn{loc}}')$ is a finitely generated $(\til{A}\otimes_{\til{A}^0} U^{-1}\til{A}^0[t_1,\dots,t_n] )^0$-module,
using \cite[Proposition 1.7(2)]{YZ1}, 
we may find a semi-free resolution
\[
\til{A}\otimes_{\til{A}^0} U^{-1}\til{A}^0[t_1,\dots,t_n] \to \til{B} \iso B_{\opn{loc}}' 
\]
such that $\til{B}^0 = (\til{A}\otimes_{\til{A}^0} U^{-1}\til{A}^0[t_1,\dots,t_n])^0$.
Using all these constructions, we obtain a commutative diagram
\[
\xymatrix{
A \ar[r]\ar[d] & A_{\opn{loc}} \ar[d]\ar[r] & A_{\opn{loc}}\ar[d] & \til{A}\ar[l]\ar[d] & \\
B \ar[r] & B_{\opn{loc}}\ar[r] & B_{\opn{loc}}' & \til{B}\ar[l]
}
\]
in which all horizontal maps are quasi-isomorphisms, 
and as $\til{B}^0 \in \thecat$, we are done.
\end{proof}

\begin{prop}\label{prop:pseudo-out-of-cat}
Let $\K$ be a Gorenstein noetherian ring of finite Krull dimension,
let $f:A \to B$ be a map in $\thecat$, 
and consider the commutative diagram
\[
\xymatrix{
A \ar[r]^{\pi_A}\ar[d]^{f} & \bar{A}\ar[d]^{\bar{f}}\\
B \ar[r]^{\pi_B} & \bar{B}
}
\]
Then there is an isomorphism
\[
(\bar{f})^{!}( \mrm{R}\opn{Hom}_A(\bar{A},R_A) ) \cong \mrm{R}\opn{Hom}_B(\bar{B},R_B)
\]
in $\cat{D}(\bar{B})$,
where $(\bar{f})^!$ is the twisted inverse image pseudofunctor from classical duality theory.
\end{prop}
\begin{proof}
If we knew that $\bar{A}, \bar{B} \in \thecat$, 
then this will follow immediately from pseudo-functoriality of $(-)^{!}$
(a fact that will be shown in Proposition \ref{prop:pseudofunctor} below),
as the horizontal maps are clearly cohomologically finite.
As we do not know if $\bar{A}, \bar{B}$ have finite flat dimension over $\K$,
we take a different route to prove this.
In the general case, 
let 
\[
\varphi:A \heqv \til{A}, \quad \psi:B \heqv \til{B}
\]
be the homological equivalences constructed in Lemma \ref{lem:hequiv-of-map}.
Let $R_{\til{A}} := \eqfun_{\varphi}(R_A)$, 
and $R_{\til{B}} := \eqfun_{\psi}(R_B)$.
Applying Proposition \ref{prop:transfer-finite} to the diagrams
\[
\xymatrix{
A \ar[r]\ar[d] & A_{\opn{loc}} \ar[d]\ar[r] & A_{\opn{loc}}\ar[d] & \til{A}\ar[l]\ar[d] & \\
\bar{A} \ar[r] & \bar{A}\ar[r] & \bar{A} & \bar{A} \ar[l]
}
\]
and
\[
\xymatrix{
B \ar[r]\ar[d] & B_{\opn{loc}} \ar[d]\ar[r] & B_{\opn{loc}}'\ar[d] & \til{B}\ar[l]\ar[d] & \\
\bar{B} \ar[r] & \bar{B}\ar[r] & \bar{B} & \bar{B} \ar[l]
}
\]
in which the arrows in the bottom rows are identity morphisms, 
we obtain isomorphisms
\[
\mrm{R}\opn{Hom}_A(\bar{A},R_A) \cong \mrm{R}\opn{Hom}_{\til{A}}(\bar{A},R_{\til{A}}),\quad
\mrm{R}\opn{Hom}_B(\bar{B},R_B) \cong \mrm{R}\opn{Hom}_{\til{B}}(\bar{B},R_{\til{B}})
\]
in $\cat{D}(\bar{A})$ and $\cat{D}(\bar{B})$ respectively. 
Since clearly $\bar{f} = \mrm{H}^0(\til{f})$,
it follows that we may assume without loss of generality that $A = \til{A}$, $B = \til{B}$ and $f = \til{f}$.

Thus, $A^0,B^0 \in \thecat$, so that $A^0$ and $B^0$ have rigid dualizing complexes $R_{A^0}$ and $R_{B^0}$. 
Moreover, $A^0 \to A$ and $B^0 \to B$ are cohomologically finite maps in $\thecat$, 
so by Corollary \ref{cor:finite}, and by uniqueness of rigid dualizing DG-modules, there are isomorphisms
\[
R_A \cong \mrm{R}\opn{Hom}_{A^0}(A,R_{A^0}), \quad R_B \cong \mrm{R}\opn{Hom}_{B^0}(B,R_{B^0})
\]
in $\cat{D}(A)$ and $\cat{D}(B)$ respectively.
Now,  consider the commutative diagram
\[
\xymatrix{
A^0 \ar[r]\ar[d]^{f^0} & \bar{A}\ar[d]^{\bar{f}}\\
B^0 \ar[r] & \bar{B}
}
\]
obtained from composing the canonical maps $A^0 \to A \to \bar{A}$ and $B^0 \to B \to \bar{B}$.
Since the horizontal maps in this diagram are finite, 
and the vertical maps are essentially of finite type,
we deduce from classical duality theory that there is an isomorphism
\[
(\bar{f})^{!}(\mrm{R}\opn{Hom}_{A^0}(\bar{A},-) \cong \mrm{R}\opn{Hom}_{B^0}(\bar{B},(f^0)^{!}(-))
\]
of functors $\cat{D}^{+}_{\mrm{f}}(A^0) \to \cat{D}^{+}_{\mrm{f}}(\bar{B})$.
Applying this isomorphism to $R_{A^0}$, and using the fact that $(f^0)^!(R_{A^0}) \cong R_{B^0}$, 
we obtain an isomorphism
\begin{equation}\label{eqn:fshrik-of-finite}
(\bar{f})^{!}(\mrm{R}\opn{Hom}_{A^0}(\bar{A},R_{A^0}) \cong \mrm{R}\opn{Hom}_{B^0}(\bar{B},R_{B^0})
\end{equation}
in $\cat{D}(\bar{B})$.
By adjunction,
\[
\mrm{R}\opn{Hom}_{A^0}(\bar{A},R_{A^0}) \cong \mrm{R}\opn{Hom}_A(\bar{A},\mrm{R}\opn{Hom}_{A^0}(A,R_{A^0})) \cong 
\mrm{R}\opn{Hom}_A(\bar{A},R_A), 
\]
and similarly,
\[
\mrm{R}\opn{Hom}_{B^0}(\bar{B},R_{B^0}) \cong \mrm{R}\opn{Hom}_B(\bar{B},\mrm{R}\opn{Hom}_{B^0}(B,R_{B^0})) \cong 
\mrm{R}\opn{Hom}_B(\bar{B},R_B).
\]
Plugging these two isomorphisms into (\ref{eqn:fshrik-of-finite}),
we deduce that
\[
(\bar{f})^{!} (\mrm{R}\opn{Hom}_A(\bar{A},R_A)) \cong \mrm{R}\opn{Hom}_B(\bar{B},R_B)
\]
as claimed.
\end{proof}

Here is the second main result of this section.

\begin{thm}\label{thm:smooth}
Let $\K$ be a Gorenstein noetherian ring of finite Krull dimension, 
and let $f:A \to B$ be a cohomologically essentially smooth map in $\thecat$.
Then
\[
R_A \otimes^{\mrm{L}}_A \Omega_{B/A}
\]
has the structure of a rigid dualizing DG-module over $B$ relative to $\K$.
\end{thm}
\begin{proof}
As in the proof of Theorem \ref{thm:gor-lifting},
let $\varphi:\til{B} \iso B$ be a K-flat resolution of $f$ over $A$,
and let $\psi: \til{A} \iso \bar{A}$ be a pseudo-finite semi-free resolution 
of the canonical map $A \to \bar{A}$ over $A$. 
Denote by $\chi$ the quasi-isomorphism $\chi:\til{A}\otimes_A \til{B} \iso \bar{B}$.
Again, as in the proof of Theorem \ref{thm:gor-lifting}, 
$\til{A}$ is K-projective over $A$, 
and $\til{A}\otimes_A \til{B}$ is K-projective over $\til{B}$.
Using Corollary \ref{cor:omega-qis}, we have that
\[
\opn{For}_{\varphi}( R_A \otimes^{\mrm{L}}_A \Omega_{B/A}) \cong R_A\otimes^{\mrm{L}}_A \Omega_{\til{B}/A},
\]
so by Proposition \ref{prop:tensor-with-forget}, we have an isomorphism
\[
\mrm{R}\opn{Hom}_{\til{B}}(B, R_A\otimes^{\mrm{L}}_A \Omega_{\til{B}/A}) \cong R_A \otimes^{\mrm{L}}_A \Omega_{B/A}.
\]
in $\cat{D}(B)$.
Since $\varphi$ is cohomologically finite, 
it follows by Corollary \ref{cor:finite} applied to $\varphi$, 
that it is enough to show that there is an isomorphism
\[
R_A \otimes^{\mrm{L}}_A \Omega_{\til{B}/A} \cong R_{\til{B}}
\]
in $\cat{D}(\til{B})$.
By Corollary \ref{cor:dual-smooth}, 
$R_A \otimes^{\mrm{L}}_A \Omega_{\til{B}/A}$ is a dualizing DG-module over $\til{B}$,
and according to \cite[Corollary 7.12]{Ye1}, the formula
\[
R \mapsto \mrm{R}\opn{Hom}_{\til{B}}(\bar{B},R)
\]
induces a bijection between isomorphism classes of dualizing DG-modules over $\til{B}$,
and isomorphism classes of dualizing complexes over $\bar{B}$. 
Thus, it is enough to show that there is an isomorphism
\[
\mrm{R}\opn{Hom}_{\til{B}}(\bar{B},R_{\til{B}}) \cong \mrm{R}\opn{Hom}_{\til{B}}(\bar{B},R_A \otimes^{\mrm{L}}_A \Omega_{\til{B}/A})
\]
in $\cat{D}(\bar{B})$. 
By Proposition \ref{prop:pseudo-out-of-cat},
and using the fact that $\bar{A} \to \bar{B}$ is essentially smooth, 
there is an isomorphism
\[
\mrm{R}\opn{Hom}_{\til{B}}(\bar{B},R_{\til{B}}) \cong \mrm{R}\opn{Hom}_A(\bar{A},R_A) \otimes^{\mrm{L}}_{\bar{A}} \Omega_{\bar{B}/\bar{A}}
\]
in $\cat{D}(\bar{B})$. 
Thus, we need to show that there is an isomorphism
\[
\mrm{R}\opn{Hom}_A(\bar{A},R_A) \otimes^{\mrm{L}}_{\bar{A}} \Omega_{\bar{B}/\bar{A}} \cong \mrm{R}\opn{Hom}_{\til{B}}(\bar{B},R_A \otimes^{\mrm{L}}_A \Omega_{\til{B}/A})
\]
in $\cat{D}(\bar{B})$. As in the proof of Theorem \ref{thm:gor-lifting},
there is an isomorphism
\[
\mrm{R}\opn{Hom}_A(\bar{A},R_A) \cong \mrm{R}\opn{Hom}_A(\til{A},R_A)\otimes^{\mrm{L}}_{\til{A}} \bar{A},
\]
so using the fact that $\Omega_{\bar{B}/\bar{A}} \cong \bar{B}\otimes^{\mrm{L}}_{\til{B}} \Omega_{\til{B}/A}$, 
we have an isomorphism
\[
\mrm{R}\opn{Hom}_A(\bar{A},R_A) \otimes^{\mrm{L}}_{\bar{A}} \Omega_{\bar{B}/\bar{A}} \cong
\mrm{R}\opn{Hom}_A(\til{A},R_A)\otimes^{\mrm{L}}_{\til{A}} \bar{B}\otimes^{\mrm{L}}_{\til{B}} \Omega_{\til{B}/A}
\]
in $\cat{D}(\bar{B})$.
Hence,
\[
\opn{For}_{\chi}
(
\mrm{R}\opn{Hom}_A(\bar{A},R_A) \otimes^{\mrm{L}}_{\bar{A}} \Omega_{\bar{B}/\bar{A}}
)
\cong \mrm{R}\opn{Hom}_A(\til{A},R_A)\otimes^{\mrm{L}}_A \Omega_{\til{B}/A}.
\]
Let $P \iso \Omega_{\til{B}/A}$ be a K-flat resolution over $\til{B}$. 
Then $P$ is also K-flat over $A$.
Moreover, because it is tilting over $\til{B}$, 
by \cite[Theorem 5.11]{Ye1} it has finite flat dimension relative to $\cat{D}(\til{B})$.
Since $\til{B}$ has $0$ flat dimension relative to $\cat{D}(A)$, 
we conclude that $P$ has finite flat dimension relative to $\cat{D}(A)$.
Hence, by Proposition \ref{prop:tensor-with-ffd}, the canonical $\til{A}\otimes_A \til{B}$-linear map
\[
\mrm{R}\opn{Hom}_A(\til{A},R_A)\otimes^{\mrm{L}}_A \Omega_{\til{B}/A} \cong
\opn{Hom}_A(\til{A},R_A)\otimes_A P \to \opn{Hom}_A(\til{A},R_A\otimes_A P)
\]
is a quasi-isomorphism. Using adjunction, we have isomorphisms in $\cat{D}(\til{A}\otimes_A \til{B})$:
\begin{eqnarray}
\opn{Hom}_A(\til{A},R_A)\otimes_A P \to \opn{Hom}_A(\til{A},R_A\otimes_A P) \cong\nonumber\\
\opn{Hom}_{\til{B}}(\til{A}\otimes_A \til{B},R_A\otimes_A P) \cong\nonumber\\
\mrm{R}\opn{Hom}_{\til{B}}(\til{A}\otimes_A \til{B},R_A\otimes^{\mrm{L}}_A \Omega_{\til{B}/A}),\nonumber
\end{eqnarray}
and since
\[
\opn{For}_{\chi} ( \mrm{R}\opn{Hom}_{\til{B}}(\bar{B},R_A \otimes^{\mrm{L}}_A \Omega_{\til{B}/A}) ) 
\cong
\mrm{R}\opn{Hom}_{\til{B}}(\til{A}\otimes_A \til{B},R_A \otimes^{\mrm{L}}_A \Omega_{\til{B}/A}),
\]
we deduce that there is an isomorphism
\[
\opn{For}_{\chi}(\mrm{R}\opn{Hom}_A(\bar{A},R_A) \otimes^{\mrm{L}}_{\bar{A}} \Omega_{\bar{B}/\bar{A}})
\cong 
\opn{For}_{\chi} ( \mrm{R}\opn{Hom}_{\til{B}}(\bar{B},R_A \otimes^{\mrm{L}}_A \Omega_{\til{B}/A}) ) 
\]
in $\cat{D}(\til{A}\otimes_A \til{B})$. 
By Proposition \ref{prop:tensor-with-forget}, this lifts to an isomorphism
\[
\mrm{R}\opn{Hom}_A(\bar{A},R_A) \otimes^{\mrm{L}}_{\bar{A}} \Omega_{\bar{B}/\bar{A}}
\cong 
 \mrm{R}\opn{Hom}_{\til{B}}(\bar{B},R_A \otimes^{\mrm{L}}_A \Omega_{\til{B}/A}) 
\]
in $\cat{D}(\bar{B})$,
and as explained above, 
this implies that
\[
R_{\til{B}} \cong R_A \otimes^{\mrm{L}}_A \Omega_{\til{B}/A}
\]
and this completes the proof.
\end{proof}

\section{The twisted inverse image pseudofunctor and perfect base change}

In this section we finally arrive to the main construction of this paper: the twisted inverse image pseudofunctor.
We continue to denote, for $A \in \thecat$, its rigid dualizing DG-module by $R_A$, 
and the corresponding duality functor $\mrm{R}\opn{Hom}_A(-,R_A)$ by $D_A$.
Let us denote by $\mathbf{DerCat}_{\K}$ the 2-category of $\K$-linear trinagulated categories.
\begin{dfn}
Let $\K$ be a Gorenstein noetherian ring of finite Krull dimension.
The twisted inverse image pseudofunctor 
\[
(-)^{!} : \thecat \to \mathbf{DerCat}_{\K}
\]
is defined as follows:
\begin{enumerate}
\item Given $A \in \thecat$,
we associate to it the category $\cat{D}^{+}_{\mrm{f}}(A)$.
\item Given a map $f:A \to B$ in $\thecat$,
if $A = B$ and $f = 1_A$, we let $f^{!}$ be the identity functor $\cat{D}^{+}_{\mrm{f}}(A) \to \cat{D}^{+}_{\mrm{f}}(A)$.
Otherwise, we set 
\[
f^{!}(-) := D_B(B\otimes^{\mrm{L}}_A D_A(-)): \cat{D}^{+}_{\mrm{f}}(A) \to \cat{D}^{+}_{\mrm{f}}(B).
\]
\item Given two maps $f:A \to B$ and $g:B \to C$ in $\thecat$, we define an isomorphism of functors 
$\phi_{f,g}: (g\circ f)^{!} \iso g^{!}\circ f^{!}$: if at least one of the maps $f,g$ is the identity, 
then $\phi_{f,g}$ is the identity natural transformation. 
Otherwise, we let $\phi_{f,g}$ be the isomorphism
\begin{eqnarray}
 (g\circ f)^{!}(M) = D_C(C\otimes^{\mrm{L}}_A D_A(M)) \iso\nonumber\\
D_C(C\otimes^{\mrm{L}}_B B\otimes^{\mrm{L}}_A D_A(M)) \iso\nonumber\\
D_C(C\otimes^{\mrm{L}}_B D_B(D_B(B\otimes^{\mrm{L}}_A D_A(M)))) = g^{!}(f^{!}(M))\nonumber
\end{eqnarray}
obtained from combining the natural isomorphisms
$C\otimes^{\mrm{L}}_B B\otimes^{\mrm{L}}_A - \cong C\otimes^{\mrm{L}}_A -$ and $- \iso D_B(D_B(-))$.
\end{enumerate}
\end{dfn}

\begin{prop}\label{prop:pseudofunctor}
Let $\K$ be a Gorenstein noetherian ring of finite Krull dimension.
Given three maps $f:A \to B$, $g:B \to C$ and $h:C \to D$ in $\thecat$,
there is an equality 
\[
\phi_{g,h} \circ \phi_{f,h\circ g} = \phi_{f,g} \circ \phi_{g\circ f,h},
\]
so that $(-)^{!} : \thecat \to \mathbf{DerCat}_{\K}$ is a pseudofunctor.
\end{prop}
\begin{proof}
Identical to \cite[Proposition 4.4]{YZ2}.
\end{proof}

Next, we study the behaivor of the twisted inverse image pseudofunctor with respect to cohomologically finite and cohomologically essentially smooth maps.

\begin{thm}\label{theorem:finite}
Let $\K$ be a Gorenstein noetherian ring of finite Krull dimension,
and let $f:A \to B$ be a cohomologically finite map in $\thecat$.
Then there is an isomorphism
\[
f^{!}(M) \cong \mrm{R}\opn{Hom}_A(B,M)
\]
of functors $\cat{D}^{+}_{\mrm{f}}(A) \to \cat{D}^{+}_{\mrm{f}}(B)$.
\end{thm}
\begin{proof}
Given $M \in \cat{D}^{+}_{\mrm{f}}(A)$, we have by Corollary \ref{cor:finite}
\begin{eqnarray}
f^{!}(M) = \mrm{R}\opn{Hom}_B(B\otimes^{\mrm{L}}_A\mrm{R}\opn{Hom}_A(M,R_A),R_B) \cong\nonumber\\
\mrm{R}\opn{Hom}_B(B\otimes^{\mrm{L}}_A\mrm{R}\opn{Hom}_A(M,R_A),\mrm{R}\opn{Hom}_A(B,R_A))\nonumber.
\end{eqnarray}
By adjunction, we have isomorphisms
\begin{eqnarray}
\mrm{R}\opn{Hom}_B(B\otimes^{\mrm{L}}_A\mrm{R}\opn{Hom}_A(M,R_A),\mrm{R}\opn{Hom}_A(B,R_A))\cong\nonumber\\
\mrm{R}\opn{Hom}_A(B\otimes^{\mrm{L}}_A\mrm{R}\opn{Hom}_A(M,R_A),R_A)\cong\nonumber\\
\mrm{R}\opn{Hom}_A(B,\mrm{R}\opn{Hom}_A(\mrm{R}\opn{Hom}_A(M,R_A),R_A)\nonumber,
\end{eqnarray}
so the result follows from the fact that $R_A$ is a dualizing DG-module.
\end{proof}

\begin{cor}\label{cor:twist-qis}
Let $\K$ be a Gorenstein noetherian ring of finite Krull dimension,
and let $f:A \to B$ be a quasi-isomorphism in $\thecat$.
Then there are isomorphisms
\[
f^{!}(M) \cong \mrm{R}\opn{Hom}_A(B,M) \cong B\otimes^{\mrm{L}}_A M
\]
of functors $\cat{D}^{+}_{\mrm{f}}(A) \to \cat{D}^{+}_{\mrm{f}}(B)$.
\end{cor}
\begin{proof}
The first isomorphism follows from Theorem \ref{theorem:finite},
and the second one from Proposition \ref{prop:tensor-with-forget}.
\end{proof}

\begin{thm}\label{theorem:smooth}
Let $\K$ be a Gorenstein noetherian ring of finite Krull dimension,
and let $f:A \to B$ be a cohomologically essentially smooth map in $\thecat$.
Then there is an isomorphism
\[
f^{!}(M) \cong M\otimes^{\mrm{L}}_A \Omega_{B/A}
\]
of functors $\cat{D}^{+}_{\mrm{f}}(A) \to \cat{D}^{+}_{\mrm{f}}(B)$.
\end{thm}
\begin{proof}
Let $A \xrightarrow{\til{f}} \til{B} \iso B$ be a K-flat resolution of $f$,
and denote the map $\til{B} \iso B$ by $b$.
By pseudofunctoriality, we have that $f^{!} = (b\circ \til{f})^{!} \cong b^{!} \circ \til{f}^{!}$.
By Corollary \ref{cor:twist-qis}, 
we have that
\[
b^{!} \circ \til{f}^{!} (M) \cong B\otimes^{\mrm{L}}_{\til{B}} \til{f}^{!} (M),
\]
and by Corollary \ref{cor:omega-qis}, 
we have that
\[
B\otimes^{\mrm{L}}_{\til{B}} \Omega_{\til{B}/A} \cong \Omega_{B/A},
\]
so we see that it is enough to prove the theorem in the case where $\til{B} = B$ is K-flat over $A$.
Using Theorem \ref{thm:smooth}, 
we have a functorial isomorphism
\begin{eqnarray}
f^{!}(M) = \mrm{R}\opn{Hom}_B(B\otimes^{\mrm{L}}_A\mrm{R}\opn{Hom}_A(M,R_A),R_B) \cong\nonumber\\
\mrm{R}\opn{Hom}_B(B\otimes^{\mrm{L}}_A\mrm{R}\opn{Hom}_A(M,R_A),R_A\otimes^{\mrm{L}}_A \Omega_{B/A})\cong\nonumber\\
\mrm{R}\opn{Hom}_A(\mrm{R}\opn{Hom}_A(M,R_A),R_A\otimes^{\mrm{L}}_A \Omega_{B/A}).\nonumber
\end{eqnarray}
Let $P \iso \mrm{R}\opn{Hom}_A(M,R_A)$ be a K-projective resolution over $A$,
and let $Q \iso \Omega_{B/A}$ be a K-flat resolution over $B$.
Since $A \to B$ is K-flat, 
$Q$ is also K-flat over $A$, 
and we have that
\[
\mrm{R}\opn{Hom}_A(\mrm{R}\opn{Hom}_A(M,R_A),R_A\otimes^{\mrm{L}}_A \Omega_{B/A}) \cong
\opn{Hom}_A(P,R_A\otimes_A Q).
\]
Since $Q$ is tilting over $B$,
it has finite flat dimension relative to $\cat{D}(B)$,
and since $A \to B$ is cohomologically smooth, 
we deduce that $Q$ has finite flat dimension relative to $\cat{D}(A)$.
Hence, by Proposition \ref{prop:tensor-with-ffd},
the $B$-linear map
\[
\opn{Hom}_A(P,R_A)\otimes_A Q \to \opn{Hom}_A(P,R_A\otimes_A Q)
\]
is a quasi-isomorphism.
Thus, we deduce that
\[
f^{!}(M) \cong \mrm{R}\opn{Hom}_A(\mrm{R}\opn{Hom}_A(M,R_A),R_A)\otimes^{\mrm{L}}_A \Omega_{B/A},
\]
so the result follows from the fact that $R_A$ is a dualizing DG-module.
\end{proof}

Recall that for a map $f:A \to B$ of DG-rings, 
we denote by $\mrm{L}f^*(-)$ the functor $-\otimes^{\mrm{L}}_A B$.
We now arrive to the last main result of this paper: derived base change for perfect maps.

\begin{thm}\label{thm:base-change}
Let $\K$ be a Gorenstein noetherian ring of finite Krull dimension,
let $f:A \to B$ be an arbitrary map in $\thecat$, 
and let $g:A \to C$ be a K-flat map in $\thecat$ 
such that $C$ has finite flat dimension relative to $\cat{D}(A)$.
Consider the induced base change commutative diagram
\[
\xymatrixcolsep{4pc}
\xymatrix{
A \ar[r]^f \ar[d]^g & B\ar[d]^{h}\\
C \ar[r]^{f'} & B\otimes_A C
}
\]
Then there is an isomorphism
\[
\mrm{L}h^* \circ f^{!}(-) \cong (f')^{!} \circ \mrm{L}g^{*}(-)
\]
of functors
\[
\cat{D}^{+}_{\mrm{f}}(A) \to \cat{D}^{+}_{\mrm{f}}(B\otimes_A C).
\]
\end{thm}
\begin{proof}
Step 1: assume first $f$ is cohomologically essentially smooth.
By Proposition \ref{prop:smooth-base-change},
$f'$ is also cohomologically essentially smooth,
and there is an isomorphism
\[
\Omega_{B\otimes_A C/C} \cong \Omega_{B/A}\otimes^{\mrm{L}}_A C.
\]
Given $M \in \cat{D}^{+}_{\mrm{f}}(A)$,
by Theorem \ref{theorem:smooth}, 
we have functorial isomorphisms
\begin{eqnarray}
\mrm{L}h^* \circ f^{!}(M) \cong (B\otimes_A C)\otimes^{\mrm{L}}_B (M\otimes^{\mrm{L}}_A \Omega_{B/A})\cong \nonumber\\
(M\otimes^{\mrm{L}}_A \Omega_{B/A})\otimes_A C \cong M\otimes^{\mrm{L}}_A \Omega_{B\otimes_A C/C}\cong\nonumber\\
(M\otimes^{\mrm{L}}_A C)\otimes^{\mrm{L}}_C \Omega_{B\otimes_A C/C} = (f')^{!}\mrm{L}g^*(M).\nonumber
\end{eqnarray}
Step 2: now, let us assume that $f$ is cohomologically finite,
and assume in addition that $f$ is K-projective. Then $f'$ is also K-projective and cohomologically finite.
Combining the K-projectivity assumptions with Theorem \ref{theorem:finite},
we have that
\begin{eqnarray}
\mrm{L}h^* f^{!}(M) \cong (B\otimes_A C)\otimes^{\mrm{L}}_B 
\mrm{R}\opn{Hom}_A(B,M) \cong \opn{Hom}_A(B,M)\otimes_A C\cong^{\diamondsuit}
\nonumber\\
\opn{Hom}_A(B,M\otimes_A C) \cong \opn{Hom}_C(B\otimes_A C,M\otimes_A C) \cong (f')^{!}\mrm{L}g^*(M).\nonumber
\end{eqnarray}
Here, the isomorphism $\diamondsuit$ follows from Proposition \ref{prop:tensor-with-ffd}.
\newline
Step 3: assume that $f:A \to B$ is a map in $\thecat$ that can be factored as
$A \xrightarrow{\varphi} D \xrightarrow{\psi} B$ where $\varphi$ is cohomologically essentially smooth,
and $\psi$ is cohomologically finite. Factor $\psi$ as $D \xrightarrow{\chi} \til{B} \xrightarrow{b} B$ where $b$ is a quasi-isomorphism and $\chi$ is K-projective. Then $\chi$ is also cohomologically finite. 
These factorizations fit into a commutative diagram in $\thecat$:
\[
\xymatrixcolsep{4pc}
\xymatrix{
A \ar[r]^{\varphi}\ar[d]^{g} & D\ar[r]^{\chi}\ar[d]^{g'} & \til{B}\ar[r]^{b}\ar[d]^{g''} & B\ar[d]^{h}\\
C \ar[r]^{\varphi'} & D\otimes_A C \ar[r]^{\chi'} & \til{B}\otimes_A C \ar[r]^{b'} & B\otimes_A C
}
\]
Notice that $b$, being a quasi-isomorphism,
is cohomologically essentially smooth.
Note also that by base change,
each of the morphisms $\varphi', \chi', b'$ is of finite flat dimension. 
Thus, step 1 and 2 apply to each of the squares in the diagram.
Hence, using pseudofunctoriality of $(-)^{!}$ and the previous steps, 
we have natural isomorphisms
\begin{eqnarray}
\mrm{L}h^* f^{!}(M) = \mrm{L}h^* (b\circ \chi \circ \varphi)^{!}(M) \cong 
\mrm{L}h^* b^{!} \chi^{!} \varphi^{!} (M) \cong\nonumber\\
(b')^{!}\mrm{L}(g'')^{*}\chi^{!} \varphi^{!} (M) \cong
(b')^{!}(\chi')^{!}\mrm{L}(g')^{*}\varphi^{!}(M) \cong\nonumber\\
(b')^{!}(\chi')^{!}(\varphi')^{!} \mrm{L}g^{*}(M) \cong (b'\circ \chi' \circ \varphi')^{!}(\mrm{L}g^{*}M) = (f')^{!}\mrm{L}g^{*}(M),\nonumber
\end{eqnarray}
and this completes this step.
\newline
Step 4: finally, let $f:A \to B$ be an arbitrary map in $\thecat$. 
Similarly to the proof of Lemma \ref{lem:hequiv-of-map},
let $U$ be the set of elements $u$ in $B^0$ such that $\pi_B(u) \in \bar{B}$ is a unit,
and let $\til{B} := U^{-1}B^0\otimes_{B^0} B$. 
The induced map $b:B \to \til{B}$ is a quasi-isomorphism.
Consider the map $\til{f} := b\circ f:A \to \til{B}$.
As the map $A^0 \to \bar{B}$ is essentially of finite type,
and because every preimage of a unit in $\bar{B}$ is a unit in $\til{B}^0$,
we see that $\til{f}$ may be factored as
\[
A \xrightarrow{\varphi} V^{-1}A^{0}[t_1,\dots,t_n]\otimes_{A^0} A \xrightarrow{\psi} \til{B},
\]
where $\varphi$ is cohomologically essentially smooth, 
and $\psi$ is cohomologically finite. 
Thus, step 3 applies to $\til{f}$.
Consider the commutative diagram
\[
\xymatrixcolsep{4pc}
\xymatrix{
A\ar[r]^{f}\ar[d]^{g} & B\ar[r]^{b}\ar[d]^{h} & \til{B}\ar[d]^{\til{h}}\\
C\ar[r]^{f'} & B\otimes_A C\ar[r]^{b'} & \til{B}\otimes_A C
}
\]
Set $\til{f'} := b'\circ f'$.
Then we have shown that there is a functorial isomorphism
\begin{equation}\label{eqn:base-change-before-forget}
\mrm{L}\til{h}^{*}(\til{f})^{!}(M) \cong (\til{f}')^{!}\mrm{L}g^*(M)
\end{equation}
in $\cat{D}(\til{B}\otimes_A C)$. 
Applying the forgetful functor $\opn{For}_{b'}$ to the left hand side of (\ref{eqn:base-change-before-forget}),
we obtain
\[
\opn{For}_{b'} ( \mrm{L}\til{h}^{*}(\til{f})^{!}(M) ) \cong
\opn{For}_{b'} ( (f^{!}(M)\otimes^{\mrm{L}}_B \til{B})\otimes_A C) \cong
f^{!}(M)\otimes_A C \cong \mrm{L}h^{*}f^{!}(M),
\]
while applying $\opn{For}_{b'}$ to the right hand side of (\ref{eqn:base-change-before-forget}), we get
\[
\opn{For}_{b'} ((\til{f}')^{!}\mrm{L}g^*(M)) \cong \opn{For}_{b'}((b')^{!}(f')^{!}\mrm{L}g^*(M)) \cong (f')^{!}\mrm{L}g^*(M).
\]
Hence, we deduce that for any $f:A \to B$ in $\thecat$, 
and for any $M \in \cat{D}^{+}_{\mrm{f}}(A)$,
there is a functorial isomorphism
\[
\mrm{L}h^* \circ f^{!}(M) \cong (f')^{!} \circ \mrm{L}g^{*}(M).
\]
\end{proof}

\begin{rem}
While the results of this paper were developed only in an affine derived setting, it is worth mentioning that  at least over ordinary schemes, by \cite[Theorem 3.2.9]{AIL2} rigid dualizing complexes can be glued under the flat topology. We expect that a similar gluing result will hold in the above derived setting, so that it would be possible to globalize the results of this paper to derived schemes and derived stacks.
\end{rem}

\textbf{Acknowledgments.}
The author would like to thank Pieter Belmans, Joseph Lipman and Amnon Yekutieli for some helpful discussions.

\end{document}